\documentclass[12pt,reqno]{amsart}
\usepackage{amsmath, amsthm, amssymb, stmaryrd}
\usepackage{hyperref}
\usepackage{enumerate}
\usepackage{url}
\usepackage{color}
\usepackage{mathtools}

\let\OLDthebibliography\thebibliography
\renewcommand\thebibliography[1]{
  \OLDthebibliography{#1}
  \setlength{\parskip}{0pt}
  \setlength{\itemsep}{0pt plus 0.3ex}
}

\usepackage{mathtools}

\usepackage{multirow, array}
\usepackage{placeins}

\usepackage{caption}
\advance \topmargin by -\headheight
\advance \topmargin by -\headsep

\evensidemargin \oddsidemargin
\newtheorem{thm}{Theorem}[section]
\newtheorem{lemma}[thm]{Lemma}
\newtheorem{prop}[thm]{Proposition}
\newtheorem{cor}[thm]{Corollary}
\newtheorem{conj}[thm]{Conjecture}

\theoremstyle{definition}
\newtheorem{defn}[thm]{Definition}

\theoremstyle{remark}
\newtheorem{remark}[thm]{Remark}

\numberwithin{equation}{section}

\newcommand{\mmod}[1]{{\,\,\mathrm{mod}\,\,#1}}

\newcommand*\wrapletters[1]{\wr@pletters#1\@nil}
\def\wr@pletters#1#2\@nil{#1\allowbreak\if&#2&\else\wr@pletters#2\@nil\fi}

\def\alp{{\alpha}} 
  
\def\gam{{\gamma}} 
\def\del{{\delta}}

\def\tet{{\theta}}  

\def\lam{{\lambda}} \def\Lam{{\Lambda}}

\def\ome{{\omega}}  
\def\eps{\varepsilon}

\def\le{\leqslant} \def\ge{\geqslant}  
\def \leq {\leqslant} \def \geq {\geqslant}

\def\d{{\,{\rm d}}}

\def \bN {\mathbb N}
\def \bP {\mathbb P}
\def \bQ {\mathbb Q}
\def \bR {\mathbb R}
\def \bZ {\mathbb Z}

\def \bI {\mathbf I}

\def \ba {\mathbf a}
\def \bb {\mathbf b}

\def \be {\mathbf e}

\def \bp {\mathbf p}
\def \bq {\mathbf q}
\def \br {\mathbf r}
\def \bs {\mathbf s}

\def \bu {\mathbf u}
\def \bv {\mathbf v}
\def \bx {\mathbf x}  \def \bX {\mathbf X}
\def \bw {\mathbf w}
\def \by {\mathbf y}   \def \bfY {\mathbf Y}
\def \bz {\mathbf z}   \def \bfZ {\mathbf Z}

\def \bzero {\mathbf 0}

\def \balp {\boldsymbol{\alp}}
\def \bbet {\boldsymbol{\beta}}

\def \bgam {\boldsymbol{\gam}}

\def \btet {\boldsymbol{\theta}}

\def \cB {\mathcal B}

\def \cI {\mathcal I}
\def \cJ {\mathcal J}

\def \cL {\mathcal L}

\def \cP {\mathcal P}

\def \cR {\mathcal R}

\def \rank {\mathrm{rank}}

\def \dim {\mathrm{dim}}

\def \det {\mathrm{det}}

\def \diag {\mathrm{diag}}

\def \supp {{\mathrm{supp}}}
\def \Bad {{\mathrm{Bad}}}
\def \dist {{\mathrm{dist}}}

\def \SL {\mathrm{SL}}
\def \GL {\mathrm{GL}}
\def \Gr {{\mathrm{Gr}}}

\begin{document}
\title[Diophantine transference inequalities]{Diophantine transference inequalities: weighted, inhomogeneous, and intermediate exponents}
\author[Chow]{Sam Chow}
\address{Department of Mathematics, University of York, Heslington, York, YO10 5DD, United Kingdom}
\email{Sam.Chow@maths.ox.ac.uk}
\author[Ghosh]{Anish Ghosh}
\address{School of Mathematics,
Tata Institute of Fundamental Research, Mumbai, India 400005}
\email{ghosh@math.tifr.res.in}
\author[Guan]{Lifan Guan}
\address{Department of Mathematics, University of York, Heslington, York, YO10 5DD, United Kingdom}
\email{lifan.guan@york.ac.uk}
\author[Marnat]{Antoine Marnat}
\address{Department of Mathematics, University of York, Heslington, York, YO10 5DD, United Kingdom}
\email{antoine.marnat@york.ac.uk}
\author[ Simmons]{David Simmons}
\address{Department of Mathematics, University of York, Heslington, York, YO10 5DD, United Kingdom}
\email{david.simmons@york.ac.uk}
\thanks{}
\date{}
\begin{abstract}
We extend the Khintchine transference inequalities, as well as a homogeneous--inhomogeneous transference inequality for lattices, due to Bugeaud and Laurent, to a weighted setting. We also provide applications to inhomogeneous Diophantine approximation on manifolds and to weighted badly approximable vectors. Finally, we interpret and prove a conjecture of Beresnevich--Velani (2010) about inhomogeneous intermediate exponents.
\end{abstract}
\maketitle

\section{Introduction}
\label{intro}

Dirichlet's approximation theorem \cite{Dir} is a foundational result in Diophantine approximation, and follows straightforwardly from the pigeonhole principle.

\begin{thm} [Dirichlet's approximation theorem] If $\btet = (\tet_1, \ldots, \tet_m) \in \bR^m$ and $N \in \bN$ then there exists $q \le N$ such that
\[
\| q \tet_i \|_{\bR / \bZ} \le N^{-1/m} \qquad (1 \le i \le m).
\]
\end{thm}

\noindent In general it is sharp, but for some $\btet$ there are closer rational approximations. This leads to the notion of \emph{exponents of Diophantine approximation}, as introduced by Khintchine \cite{Khi1926} and Jarn\'ik \cite{Jar1938}. In this article, we concern ourselves with the very general setting of weighted Diophantine exponents, uniform or otherwise.

\begin{defn}\label{def-1}%add names in full letters % I'm happy with the names of the exponents being included, but I've re-edited because some of this was ungrammatical.
Let $m,n \in \bN$ denote dimensions, and let $\bs=(s_1,\ldots,s_m) \in \bR_{>0}^m$ and $\br=(r_1,\ldots, r_n)\in \bR_{>0}^n$ be weights, that is
\begin{equation*}
\sum_{i=1}^{m}s_m = \sum_{j=1}^{n}r_n = 1.
\end{equation*}
Write
\begin{equation*}
  \|(x_1,\ldots,x_m)\|_{\bs}=\max_{1\le i\le m}|x_i|^{\frac{1}{s_i}},\qquad \|(y_1,\ldots,y_n)\|_{\br}=\max_{1\le j \le n}|y_j|^{\frac{1}{r_j}}.
\end{equation*}
Let $A\in M_{m\times n}(\bR)$ and $\btet\in \bR^m$. The \emph{inhomogeneous weighted exponent}, denoted $\omega_{\bs,\br}(A,\btet)$, is the supremum of the real numbers $\omega$ for which, for some arbitrarily large real numbers $T$, the inequalities
\begin{equation}\label{equ-def-w}
  \|\bq\|_{\br}<T, \qquad \|A\bq-\bp-\btet\|_{\bs}<T^{-\omega}
\end{equation}
have a  solution $(\bp,\bq)\in \bZ^m\times(\bZ^n\setminus \{\bzero\})$.
The \emph{uniform inhomogeneous weighted exponent}, denoted $\hat{\omega}_{\bs,\br}(A,\btet)$, is the supremum of the real numbers $\hat{\omega}$ for which, for \textbf{all} sufficiently large real numbers $T$, the inequalities
\begin{equation*}
  \|\bq\|_{\br}<T, \qquad \|A\bq-\bp-\btet\|_{\bs}<T^{-\hat{\omega}}
\end{equation*}
have a solution $(\bp,\bq)\in \bZ^m\times(\bZ^n\setminus \{\bzero\})$. Moreover, define the homogeneous exponents
\[
\omega_{\bs, \br}(A) = \omega_{\bs, \br}(A, \bzero), \qquad \hat \omega_{\bs, \br}(A) = \hat \omega_{\bs, \br}(A,\bzero).
\]
Finally, in the unweighted case $\bs=(1/m,\ldots, 1/m)$ and $\br=(1/n,\ldots, 1/n)$, write $\omega(A,\btet)$ and $\hat{\omega}(A,\btet)$ for $\omega_{\bs,\br}(A,\btet)$ and $\hat{\omega}_{\bs,\br}(A,\btet)$, respectively.
\end{defn}

\begin{remark} \label{normalisation} Diophantine exponents are allowed to equal $+\infty$. We have normalised in such a way that Dirichlet's approximation theorem delivers the lower bound $1$ for the exponent, as in \cite{BV2010}; it is more common to normalise so that Dirichlet's approximation theorem delivers the lower bound $n/m$. This value is said to be \emph{critical}, as it is attained by almost all matrices $A$, especially in the context of \S \ref{S2}.
\end{remark}

Exponents of multiplicative Diophantine approximation can be similarly defined.

\begin{defn}
Let $A\in M_{m\times n}(\bR)$ and $\btet\in \bR^m$. The \emph{inhomogeneous multiplicative exponent}, written $\omega^{\times}(A,\btet)$, is the supremum of the real numbers $\omega$ for which, for some arbitrarily large real numbers $T$, the inequalities
\begin{equation*}
  \Pi_{+}(\bq)<T, \qquad \Pi (A\bq-\bp-\btet)<T^{-\omega}
\end{equation*}
have a solution $(\bp,\bq)\in \bZ^m\times(\bZ^n\setminus \{\bzero\})$, where
\[\Pi_{+}(\bq)=\prod_{j=1}^{n}\max\{1, |q_j|\},\qquad \Pi(\by)=\prod_{i=1}^{m}|y_i| \]
with $\by=(y_1,\ldots, y_m)$.
The \emph{uniform inhomogeneous multiplicative exponent}, written $\hat{\omega}^{\times}(A,\btet)$, is the supremum of the real numbers $\hat{\omega}$ for which, for \textbf{all} sufficiently large real numbers $T$, the inequalities
\begin{equation*}
  \Pi_{+}(\bq)<T, \qquad \Pi (A\bq-\bp-\btet)<T^{-\hat{\omega}}
\end{equation*}
have a solution $(\bp,\bq)\in \bZ^m\times(\bZ^n\setminus \{\bzero\})$. The homogeneous multiplicative exponents are
\[
\omega^{\times}(A) := \omega^\times(A,\bzero), \qquad \hat{\omega}^{\times}(A) := \hat{\omega}^\times(A,\bzero).
\]
\end{defn}

\begin{remark}\label{rem-1}
 It follows directly from the definitions that for any weights $(\bs,\br)$ and any $(A,\btet)$ we have
\[
\omega^{\times}(A,\btet)\ge \omega_{\bs, \br}(A,\btet),
\qquad \hat{\omega}^{\times}(A,\btet)\ge \hat{\omega}_{\bs, \br}(A,\btet),
\]
and in particular
\[ \omega^{\times}(A)\ge \omega_{\bs, \br}(A), \qquad \hat{\omega}^{\times}(A)\ge \hat{\omega}_{\bs, \br}(A). \]
\end{remark}

It is well-known that, for  $A \in M_{m \times n}(\bR)$, the Diophantine exponent $\omega(A)$ and that of its transpose $\omega({}^tA)$ are related by \emph{transference inequalities}. This was first observed by Khintchine \cite{Khi1926} in the $n=1$ (simultaneous approximation) or $m=1$ (dual approximation) cases. The following generalisation is due to Dyson \cite[Theorem 4]{Dyson}, see also \cite[Chapter 6, \S 45, Theorem 8]{GL} and \cite[Chapter IV, \S 5]{Schmidt}.

\begin{thm} \label{DysonThm}
  Let $A \in M_{m \times n}(\bR)$. Write $\omega=\omega(A)$ and ${}^t\omega=\omega({}^t{A})$. Then
  \[
{}^t\omega \ge \frac{n\omega+m-1}{(n-1)\omega+m} \quad\text{ and }\quad \omega\ge \frac{m \: {}^t\omega+n-1}{(m-1){}^t\omega+n}.
\]
  In particular, we have $\omega=1$ if and only if ${}^t\omega=1$. The inequalities also hold for the uniform exponents with $\omega=\hat{\omega}(A)$ and ${}^t\omega=\hat{\omega}({}^t{A})$.
\end{thm}

We extend Dyson's transference inequalities to the weighted case. Set
\begin{equation*}
  \rho_{\bs}=\max_{1\le i\le m}s_i,\quad \delta_{\bs}=\min_{1\le i\le m}s_i, \quad\text{ and }\quad \rho_{\br}=\max_{1\le j\le n}r_j,\quad \delta_{\br}=\min_{1\le j\le n}r_j.
\end{equation*}
\begin{thm}\label{thm-khin}
 Let $A \in M_{m \times n}(\bR)$. Write $\omega=\omega_{\bs,\br}(A)$ and ${}^t\omega=\omega_{\br, \bs} ({}^t{A})$. Then
 \begin{equation*}
   {}^t\omega\ge \frac{(m+n-1)\rho_{\bs}\rho_{\br}(\delta_{\br}+\delta_{\bs}\omega)
   +\rho_{\bs}\delta_{\br}\delta_{\bs}(\omega-1)}{(m+n-1)\rho_{\bs}\rho_{\br}(\delta_{\br}
   +\delta_{\bs}\omega)-\rho_{\br}\delta_{\br}\delta_{\bs}(\omega-1)}
 \end{equation*}
 and
 \begin{equation*}
   \omega\ge \frac{(m+n-1)\rho_{\bs}\rho_{\br}(\delta_{\bs}+\delta_{\br}{}^t\omega)
   +\rho_{\br}\delta_{\br}\delta_{\bs}({}^t\omega-1)}{(m+n-1)\rho_{\bs}\rho_{\br}(\delta_{\bs}
   +\delta_{\br}{}^t\omega)-\rho_{\bs}\delta_{\br}\delta_{\bs}({}^t\omega-1)}.
 \end{equation*}
 In particular, $\omega=1$ if and only if ${}^t\omega=1$. The inequalities also hold for the uniform exponents with $\omega=\hat{\omega}_{\bs,\br}(A)$ and ${}^t\omega=\hat{\omega}_{\br, \bs}({}^t{A})$.
\end{thm}

%\noindent Cassels observed that the natural inhomogeneous analogue of Dirichlet's approximation theorem is false %\cite[Ch. 3, Theorem III]{Cas1}. It is also classical knowledge that inhomogeneous problems in Diophantine %approximation are connected to homogeneous ones via transference theorems \cite{Cas1}. The reason is that the %difference between two inhomogeneous approximations constitutes a homogeneous one.\\

Our next result concerns the following elegant result due to Bugeaud and Laurent \cite{BL2005}.
\begin{thm}[Bugeaud--Laurent] \label{BL}
For $A \in M_{m \times n}(\bR)$ and $\btet \in \bR^n$ we have
\[
\omega(\prescript{t}{}{A}, \btet) \ge \hat{\omega}(A)^{-1}, \qquad \hat{\omega}(\prescript{t}{}{A},\btet) \ge \omega(A)^{-1},
\]
with equality for Lebesgue-almost all $\btet$.
\end{thm}

The proof in \cite{BL2005} fails to deliver a weighted version. Using a slightly different proof, we are able to establish the following more general assertion.

\begin{thm} \label{MainThm}
If $A \in M_{m \times n}(\bR)$ and $\btet \in \bR^n$ then
\begin{equation}\label{equ-main-ine}
  \omega_{\br, \bs}( \prescript{t}{}{A}, \btet) \ge \hat{\omega}_{\bs,\br}(A)^{-1},\qquad \hat{\omega}_{\br,\bs}( \prescript{t}{}{A},\btet) \ge \omega_{\bs,\br}(A)^{-1},
\end{equation}
with equality for Lebesgue-almost all $\btet$.
\end{thm}

\noindent The inequalities \eqref{equ-main-ine} still hold for the natural multiplicative analogue, in view of Remark \ref{rem-1}, as was already noted in \cite{GM2018}. However, if $\omega^{\times}(A)> \omega_{\bs, \br}(A)$ then equality is never attained.

\begin{remark} The homogeneous case $\btet = \bzero$ can be quickly seen as follows. There is a weighted form of Dirichlet's approximation theorem (whose proof is essentially the same, using Minkowski's first convex body theorem) which implies that $\ome_{\bs,\br}(A) \ge \hat \ome_{\bs, \br}(A) \ge 1$. Hence
\[
\ome_{\br, \bs}(^tA) \ge 1 \ge \hat \ome_{\bs, \br}(A)^{-1}.
\]
\end{remark}

In the course of the proof of Theorem \ref{MainThm}, we extend the theory of best approximations to a weighted setting in \S \ref{S4}. As a by-product, we obtain results on the dimension of a certain set of inhomogeneous shifts for weighted $\eps$-badly approximable matrices, namely a weighted generalisation of \cite[Theorem 1.5]{BKLR}. 

\begin{thm}\label{thm-bklr}
Let $A \in M_{m \times n}(\bR)$ be a matrix for which the group $G := {^tA} \bZ^m + \bZ^n$ has rank 
\[
\rank_\bZ(G) = m + n.
\]
Let $(\bq_k)_{k \ge 1}$ be a sequence of weighted best approximations associated to ${}^tA$, and suppose  $\lim_{k\rightarrow \infty}\|\bq_k\|_{\br}^{1/k}=\infty$. For $\eps > 0$, define
\[
\Bad_{\br,\bs}^{\eps}({}^tA) = \{ \btet\in \bR^n: \liminf_{(\bp,\bq)\in (\bZ^{m}\setminus\{\bzero\})\times\bZ^n} \|\bp\|_{\bs}\|{}^tA\bp-\bq-\btet\|_{\br}\ge \eps \}.
\]
Then there exists $\eps = \eps(A)>0$ such that
\[
\dim_H \Bad_{\br,\bs}^{\eps}({}^tA)=n.
\]
%Moreover, the constant $\eps$ can be taken to be absolute if $\lim_{i\rightarrow \infty} %\|\bq_{i+1}\|_{\br}/\|\bq_{i}\|_{\br}=\infty$
\end{thm}

\noindent This is often referred to as \emph{twisted} diophantine approximation: the inhomogeneous shift is metric. The analogous problem for weighted badly approximable matrices has hitherto been investigated in \cite{HM2017, BM2017}; therein the object of study is $\displaystyle \Bad_{\br,\bs}({}^tA) := \cup_{\eps>0}\Bad_{\br,\bs}^{\eps}({}^tA)$. Our conclusion is stronger than the assertion that $\displaystyle \dim_H \Bad_{\br,\bs} (^t A) = n$.

\begin{remark} 
As noted in \cite{BKLR} and prior works, if the maximal rank condition is not met then $^tA\bx = \bzero$ possesses infinitely many solutions $\bx \in \bZ^m$, and the theory of best approximations breaks down. This excluded case is less interesting: the conclusion is still valid, since any $\btet \in \Bad_{\br, \bs}^\eps(^tA)$ would need to lie close to a discrete family of parallel hyperplanes in $\bR^n$. The rank condition will be discussed further in \S \ref{S5}, in a similar context.
\end{remark}

\bigskip

When $m=1$ or $n=1$, the matrix $A$ may be interpreted as a vector. In these special cases, the exponents defined above are the classical exponents of simultaneous and dual approximation; as discussed, these are related by Khintchine's transference inequality \cite{Khi1926}. In \cite{Laurent}, Laurent introduced intermediate exponents, refining the above quantities.

In a nutshell, the $d$th intermediate exponent quantifies a vector's proximity to $d$-dimensional rational subspaces. We can write down a $d$-dimensional linear subvariety $\cL$ of $\mathbb P_\bR^n$ using homogeneous coordinates. These span a $(d+1)$-dimensional subspace $V$ of $\bR^{n+1}$, and $\cL = \bP(V)$ is \emph{rational} if $V$ has a rational basis, that is, a basis
\[
\{ \bx_0, \ldots, \bx_d \} \subset \bQ^{n+1}.
\]
To define the height $H(\cL)$ of a $d$-dimensional rational linear subvariety, Schmidt \cite{SchH} began by using the Pl\"ucker embedding
\[
\mathrm{Gr}(d, \bP_\bR^n) \hookrightarrow \bP_\bR^{ {n+1 \choose d+1}-1}
\]
to obtain Grassmannian coordinates for $\cL$. Explicitly, this yields
\[
\cL \mapsto \bX := \bx_0 \wedge \cdots \wedge \bx_d \in \bP_\bR^{{n+1 \choose d+1}-1}
\]
and, since the basis is rational, we in fact have $\bX \in \bP_\bQ^{{n+1 \choose d+1}-1}$. The \emph{height} $H(\cL)$ of $\cL$ is the Weil height $|\bX|$ of $\bX$ (one rescales the projective coordinates to obtain a primitive integer vector, then evaluates the supremum norm). Schmidt did not work projectively; in this aspect we follow Laurent \cite{Laurent}.

The distance generalises the notion of \emph{projective distance} between two points \cite{Beresnevich, Philippon, Roy}. Recall that there is a unique defined inner product on $\Lam^t(\bR^{n+1})$ such that for any two multivectors 
\[
\bu = \bu_1 \wedge \cdots \wedge \bu_t \text{ and } \bv = \bv_1 \wedge \cdots \wedge \bv_t
\]
we have 
\begin{equation}\label{def-inner}
\langle \bu, \bv \rangle=  \det(\langle \bu_i,  \bv_j \rangle)_{i,j=1}^t,
\end{equation}
see \cite[\S 3]{Beresnevich}.
Then the \emph{Euclidean norm} $|\bu|$ of a multivector
\[
\bu = \bu_1 \wedge \cdots \wedge \bu_t
\]
is given by
\[
|\bu|^2 = |\det(\langle \bu_i,  \bu_j \rangle)_{i,j=1}^t|.
\]
 For $P,Q \in \bP_\bR^n$, the \emph{projective distance} between $P$ and $Q$ is
\[
\d(P,Q) = \frac{ | \bx \wedge \by| } {|\bx| \cdot | \by|},
\]
where $\bx$ and $\by$ are homogeneous coordinates for $P$ and $Q$ respectively. If $P \in \bP_\bR^n$ then
\begin{equation} \label{min}
\d (P, \cL) := \displaystyle \min_{Q \in \cL} \d(P, Q)
\end{equation}
is the least projective distance between $P$ and a point of $\cL$.

\begin{defn} \label{InterDef}
Let $d$ be an integer in the range $0\leq d \leq n-1$, and let $\balp \in \bR^n$. Define the \emph{$d$th ordinary exponent} ${\omega}_{d}(\balp)$  (resp. the \emph{$d$th uniform exponent} $\hat{\omega}_{d}(\balp)$) as the supremum of the real numbers $\omega$ for which there exist $d$-dimensional rational linear subspaces $\cL \subset \mathbb{R}^{n}$ such that
 \[  H(\cL)^{n \choose d} \leq T, \qquad (T \cdot \d([1:\balp], \cL))^{n \choose d+1} \leq T^{-\omega}  \]
for some arbitrarily large real numbers $T$  (resp. for every sufficiently large real number $T$).
\end{defn}

\noindent Here we have chosen the normalisation with powers ${n \choose d}$ and ${n \choose d+1}$ so that the exponent is generically $1$. The cases $d=0$ and $d=n-1$ correspond to the simultaneous and dual cases, respectively, see \cite{Laurent}.

In \cite{BV2010} an associated inhomogeneous Diophantine exponent $\omega_d(\balp, \btet)$ is posited but not defined, where
\[
\balp \in \bR^n, \qquad d \in \{ 0,1,\ldots, n-1 \}, \qquad \btet \in \bR^{n - d}.
\]
Then, the following transference inequality is conjectured \cite[Conjecture 3]{BV2010}.

\begin{conj}[Beresnevich--Velani] \label{BVconj}
Let $\balp \in \bR^n$ and $d \in \{0, \ldots , n-1\}$. Then for all $\btet \in \bR^{n-d}$ we have
\[
\omega_d(\balp ,\btet) \geq \cfrac{1}{\hat{\omega}_{n-1-d}(\balp)}.
\]
\end{conj}

In \cite[\S 2]{Laurent}, Laurent introduced an equivalent definition of $\dist(\balp, \cL)$. We will use this to formally define the inhomogeneous intermediate exponents $\ome_d(\balp, \btet)$ in \S \ref{S6}, and establish the resulting interpretation of Conjecture \ref{BVconj}. This applies to a shift $\btet \in \bR^{n \choose d+1}$.

\begin{thm} \label{ourBV}
Let $\balp \in \bR^n$ and $d \in \{0, \ldots , n-1\}$. Then for $\btet \in \bR^{n \choose d+1}$ we have
\[ 
\omega_d(\balp,\btet) \geq \cfrac{1}{\hat{\omega}_{n-1-d}(\balp)}, 
\]
with equality for Lebesgue-almost all $\btet$.
\end{thm}

\noindent
After working with the definitions, we will see that this follows directly from Theorem \ref{BL}.
%\noindent Our shift is more general than what is common in inhomogeneous Diophantine approximation. For example, %when $n=1$ and $d = 0$, the pertinent system is
%\[
%\max \{ |q|, |a_1| \} \le T, \qquad | (q + \tet_0) \alp - (a_1 +\tet_1)| \le T^{- \ome}.
%\]
%It is more standard to distinguish between numerators and denominators by setting $\tet_0 = 0$. By not doing this, %we obtain a more general statement, and one that is more natural from an algebraic or dynamical standpoint. We will %find that Theorem \ref{ourBV} is a property of exterior algebras.

\iffalse
\noindent
The homogeneous case is $\btet = \bzero$ of this may be deduced from the work of Laurent \cite{Laurent}. Indeed, the corollary to \cite[Theorem 2]{Laurent} reveals that
\[
\omega_d(\bx) \geq \cfrac{d+1}{n-d}.
\]
\fi

\subsection*{Organisation.} In \S \ref{S3}, we use the geometry of numbers to establish a property for general approximating functions, which would imply \eqref{equ-main-ine}. Then, in \S \ref{S3-1}, we prove Theorem \ref{thm-khin}. In \S \ref{S4}, we extend the theory of best approximations to a weighted setting. This allows us to finish the proof of Theorem \ref{MainThm} in \S \ref{S5}. Theorem \ref{thm-bklr} is established in \S \ref{S5-1}.  In \S \ref{S2}, we discuss applications to the theory of inhomogeneous Diophantine approximation on manifolds. Finally, in \S \ref{S6}, we define $\omega_d(\bx,\btet)$ and prove Theorem \ref{ourBV}.

\subsection*{Funding and acknowledgments} SC, LG, AM and DS were supported by EPSRC Programme Grant EP/J018260/1. AM was also supported by FWF Project I 3466-N35 and FWF START Project Y-901.  AG was supported by the Indo-French Centre for the Promotion of Advanced Research; a Department of Science and Technology, Government of India Swarnajayanti fellowship; a MATRICS grant from the Science and Engineering Research Board; and the Benoziyo Endowment Fund for the Advancement of Science at the Weizmann Institute. AG gratefully acknowledges the hospitality of the Technion and the Weizmann Institute. We thank Yann Bugeaud for an inspiring question and for fruitful discussions, Matthias Schymura for helpful comments regarding Lemma \ref{geo-of-num}, and an anonymous referee for carefully reading the manuscript and suggesting a few small changes.

\section{The geometry of numbers}
\label{S3}

For the proof of Theorem \ref{MainThm}, the following lemma is pivotal, and we anticipate that it will find uses in other contexts. Evertse has recently pointed out in a survey article \cite{Eve2018} that this lemma is implicit in Mahler's work \cite{Mah1939} from the late 1930s (in German). For completeness, we supply the details below.

\begin{lemma}\label{geo-of-num}
Let $d \in \bN$, and let $C = C_d = d! (3/2)^{\frac{d-1}2}d$. Let $\Lam$ be a full lattice in $\bR^d$, and let $\cR \subseteq \bR^d$ be a symmetric, convex body such that $\cR \cap \Lam = \{ \bzero \}$. Then for all $\bgam \in \bR^d$ we have
\[
(C \cR^* + \bgam) \cap \Lam^* \ne \emptyset,
\]
and moreover
\[
C \cR^*  \cap (\Lam^*\setminus\{\bzero\}) \ne \emptyset.
\]
Here and throughout, we use an asterisk to denote a dual/polar convex set or lattice \cite{Cas2}. 
\end{lemma}

\begin{proof}
By assumption, the first successive minimum satisfies $\mu_1(\cR) > 1$, and so \cite[Ch. VIII, Theorem VI]{Cas2} gives
\begin{equation} \label{mud}
\mu_d(\cR^*) < d!.
\end{equation}
Now the inequality at the end of the proof of the First Finiteness Theorem in \cite[Lecture X, \S 6]{Sie1989} implies that $\frac C d \cR^*$ contains a basis $\bb_1^*, \ldots, \bb_d^*$ for $\Lam^*$. Hence $C \cR^*$ contains $d\bb_1^*, \ldots, d \bb_d^*$, as well as the origin, and therefore also contains the convex hull of these vectors, which in turn contains the fundamental parallelepiped
\[
\{ w_1 \bb_1^* + \cdots + w_d \bb_d^*: \bw \in [0, 1]^d \}
\]
for $\Lam^*$. Therefore any translate of $C \cR^*$ intersects $\Lam^*$. The second assertion follows directly from \eqref{mud}.
\end{proof}

\begin{remark} The reader may consult \cite{KL1988, Ban1996} for sharper and more general results in this direction.
\end{remark}

Lemma \ref{geo-of-num} enables us to tackle transference inequalities for more general approximation functions.
\begin{defn}
Let $\psi: \bR_{>0}\rightarrow \bR_{>0}$ be a strictly decreasing function. A pair $(A,\btet) \in M_{m \times n}(\bR) \times \bR^n$ is \emph{$(\psi,\bs,\br)$-approximable} (resp. \emph{uniformly $(\psi,\bs,\br)$-approximable}) if for some arbitrarily large (resp. all sufficiently large) real numbers $T$ the inequalities
\begin{equation}\label{equ-psi}
  \|\bq\|_{\br}<T, \qquad \|A\bq-\bp-\btet\|_{\bs}<\psi(T)
\end{equation}
have a solution $(\bp,\bq)\in \bZ^m\times(\bZ^n\setminus \{\bzero\})$. The matrix $A \in M_{m \times n}(\bR)$ is \emph{$(\psi,\bs,\br)$-approximable} (resp. uniformly $(\psi,\bs,\br)$-approximable) if this holds for $(A,\bzero)$. In the special case $\psi(T)=T^{-\omega}$, we also write (uniformly) $(\omega,\bs,\br)$-approximable to mean (uniformly) $(\psi, \bs, \br)$-approximable.
\end{defn}

\begin{lemma}\label{lem-psi}
Let $A \in M_{m \times n}(\bR)$ and $\btet \in \bR^n$. Let $\psi, \phi: \bR_{>0}\rightarrow \bR_{>0}$ be strictly decreasing functions with \begin{equation}\label{equ-lim}
       \lim_{T\rightarrow \infty}\psi(T)=\lim_{T\rightarrow \infty}\phi(T)=0,
       \end{equation}
 and suppose that if $C$ and $T$ are sufficiently large then
  \begin{equation}\label{equ-psi-phi}
  \phi(C\psi(T)^{-1})>CT^{-1}.
  \end{equation}
If $A$ is not $(\psi,\bs,\br)$-approximable, then $(^tA,\btet)$ is uniformly $(\phi, \br, \bs)$-approximable for all $\btet\in \bR^n$. If $A$ is not uniformly $(\psi,\bs,\br)$-approximable, then $(^tA,\btet)$ is $(\phi, \br, \bs)$-approximable for all $\btet\in \bR^n$.
\end{lemma}

\begin{proof}
  For $T>0$, write
  \[g_{\br}(T)=\diag(T^{r_1},\ldots,T^{r_n})\in \GL_{n}(\bR),\]
and define $g_{\bs}(T)\in \GL_{m}(\bR)$ similarly. For $Q, T>0$  and $A\in M_{m\times n}(\bR)$, we define a full lattice in $\bR^{m+n}$ as follows,
  \[\Lambda(Q, T, A)=\begin{pmatrix} g_{\bs}(Q^{-1}) &  \\ & g_{\br}(T^{-1}) \end{pmatrix} \begin{pmatrix} \bI_m & A \\  & \bI_n \end{pmatrix} \bZ^{m+n}.\]
Observe that \eqref{equ-psi} has a non-zero integer solution with $\btet = \bzero$ if and only if
  \[ \Lambda(\psi(T), T, A)\cap \cB\ne \{ \bzero \},\]
where $\cB=[-1,1]^{m+n}$. The dual region is
\[
\cB^*=\left\{(x_1,\ldots,x_{m+n})\in \bR^{m+n}: \sum_{i=1}^{m+n}|x_i|\le 1\right\},
\]
and in particular $\cB^*\subset\cB$.

We now set about proving the first assertion. The proof of the second statement is similar, and omitted. Let $T$ be a large positive real number, and let $\btet \in \bR^n$. If $A$ is not $(\psi,\bs,\br)$-approximable then \eqref{equ-psi} has no solution with $\btet = \bzero$. In light of the discussion above, we have
\[
\Lambda(\psi(T), T, A)\cap \cB = \{ \bzero \},
\]
and now Lemma \ref{geo-of-num} yields
   \begin{equation}\label{equ-intersetion}
    \Lambda(\psi(T), T, A)^* \cap (C\cB+\bgam) \ne \emptyset \qquad( \bgam \in \bR^{m+n}).
   \end{equation}

A standard calculation gives
    \[
\Lambda(\psi(T), T, A)^*=\begin{pmatrix} g_{\bs}(\psi(T)) &  \\ & g_{\br}(T) \end{pmatrix} \begin{pmatrix} \bI_m &  \\  -^tA & \bI_n \end{pmatrix} \bZ^{m+n}.
\]
Applying
    \eqref{equ-intersetion} with $\bgam=\bgam_{T,\btet}:=(\bzero, g_{\br}(T)\btet)\in \bR^{m+n}$, we find that the inequalities
\begin{equation} \label{GeneralInequalities}
\|\bp\|_{\bs}<C_1 \psi(T)^{-1}, \qquad \|^tA\bp-\bq-\btet\|_{\br}<C_1T^{-1}
\end{equation}
have a solution $(\bp,\bq)\in \bZ^{m+n}$. Here $C_1 = C^{\frac1 {\delta_{\bs}}}$.

We claim that $\bp$ can be chosen to be non-zero. There are two cases to consider. When $\btet \notin \bZ^n$, we know that $\|\bq-\btet\|_{\br}$ is bounded away from 0, so as $T$ is large the second inequality of \eqref{GeneralInequalities} cannot be satisfied when $\bp=\bzero$. On the other hand,  when $\btet \in \bZ^n$, we may freely suppose that $\btet = \bzero$, whereupon the second part of Lemma \ref{geo-of-num} allows us to take $(\bp, \bq) \ne (\bzero,\bzero)$, and then the largeness of $T$ forces $\bp \ne \bzero$. Using \eqref{equ-lim} and \eqref{equ-psi-phi}, we finally conclude that if $T_1$ is large then the inequalities
    \[ \|\bp\|_{\bs}<T_1, \qquad \|^tA\bp-\bq-\btet\|_{\br}<\phi(T_1)\]
    have a  solution $(\bp,\bq)\in (\bZ^m\setminus \{\bzero\})\times \bZ^n$. Therefore $(^tA,\btet)$ is uniformly $(\phi, \br, \bs)$-approximable.
\end{proof}

\section{Dyson's theorem with weights} \label{S3-1}

The purpose of this section is to prove Theorem \ref{thm-khin}. As in the unweighted case, this essentially follows from Minkowski's second convex body theorem.

By symmetry, it suffices to prove the first assertion of the theorem. For $t_1, t_2\in \bR$, set
\[L(t_1, t_2, A)=\Lambda(e^{t_1}, e^{t_2}, A),  \qquad  h(t_1, t_2)=\diag(g_{\bs}(e^{t_1}), g_{\br}(e^{t_2})).\]
  Let $v<\omega$. By definition, for some arbitrarily large $t$, we have
  \[L(-vt, t, A) \cap \cB \ne \{\bzero\}, \]
where $\cB = [-1,1]^{m+n}$, or equivalently
  \[h\left(\delta_{\br}(\delta_{\bs}+\delta_{\br})^{-1}(1-v)t, \delta_{\bs}(\delta_{\bs}+\delta_{\br})^{-1}(1-v)t\right)
  \big(L(-vt, t, A) \cap \cB\big) \ne \{\bzero\}.\]
It then follows that
\[L(-t_1, t_1, A)\cap h\left(\delta_{\br}(\delta_{\bs}+\delta_{\br})^{-1}(1-v)t, \delta_{\bs}(\delta_{\bs}+\delta_{\br})^{-1}(1-v)t\right)\cB \ne \{\bzero\},\]
where $t_1=(\delta_{\bs}+\delta_{\br})^{-1} (\delta_{\br}+\delta_{\bs}v)t$.
Hence
\[\mu_1(L(-t_1, t_1, A), \cB)\le e^{-t_0},\]
where \[t_0=(\delta_{\bs}+\delta_{\br})^{-1}\delta_{\bs}\delta_{\br}(v-1)t= (\delta_{\bs}v+\delta_{\br})^{-1}\delta_{\bs}\delta_{\br}(v-1)t_1.\]
By \cite[Ch. VIII, Theorem VI]{Cas2}, we now have
\[\mu_{m+n}(L(-t_1, t_1, A)^*, \cB)\ge  e^{t_0}.\]
Minkowski's second convex body theorem \cite[Appendix B, Theorem V]{Cas1} now gives
\[\mu_{1}(L(-t_1, t_1, A)^*, \cB)\le C e^{-(m+n-1)^{-1}t_0},\]
for some constant $C > 0$ depending only on $m+n$, and so
\[\begin{pmatrix} g_{\bs}(e^{-t_1}) &  \\ & g_{\br}(e^{t_1}) \end{pmatrix} \begin{pmatrix} \bI_m &  \\ -{}^tA & \bI_n \end{pmatrix} \bZ^{m+n} \cap e^{-(m+n-1)^{-1}t_0+\log C }\cB \neq \{ \bzero \}.\]
Therefore
\[L(-t', t'', {}^tA)\cap \cB \neq \{ \bzero \},\]
where
\[t'=t_1-\rho_{\bs}^{-1}((m+n-1)^{-1}t_0-\log C ) \text{ and } t''=t_1+\rho_{\br}^{-1} ((m+n-1)^{-1}t_0-\log C ).\]
Since $C $ is constant, we get
\[{}^t\omega\ge \frac{1+\rho_{\br}^{-1}t_1^{-1}(m+n-1)^{-1}t_0}{1-\rho_{\bs}^{-1}t_1^{-1}(m+n-1)^{-1}t_0}.\]
As $v<\omega$ is arbitrary, a direct calculation completes the proof of Theorem \ref{thm-khin}.\\

If we consider uniform exponents, the proof is the very same by choosing \textbf{any} sufficiently large $t$, instead of \textbf{some} arbitrarily large $t$.

\begin{remark}
We do not know whether Theorem \ref{thm-khin} is optimal for every choice of weights. The special case in which the weights are uniform, namely Theorem \ref{DysonThm}, is however known to be optimal: Jarn\'ik \cite{Jar1959} established this in quite some generality, for example if
\[
1 \le m \le n, \qquad 1 \le \ome \le \infty
\]
then the first inequality in Theorem \ref{DysonThm} is sharp. All of this is discussed more broadly after the proof of \cite[Chapter 6, \S 45, Theorem 8]{GL}; the reader should be wary of the difference in normalisation therein.
\end{remark}

\section{Best approximations} \label{S4}

When working in higher dimensions, the theory of \emph{best approximations} often acts as a proxy for the theory of continued fractions. Best approximations were introduced by Voronoi \cite{Vor} as minimal points in lattices, and Rogers \cite{Rog} was the first to define them in the context of exponents of Diophantine approximation. We require a weighted version of the best approximations employed in \cite[\S 3]{BL2005}. The properties presented therein generalise cleanly to a broad setting, which includes the weighted case. We supply full details for completeness, closely following \cite{BL2005}.

\begin{defn}
Let $\Lambda$ be a lattice in a real vector space.
Let $N,L: \Lambda \to [0,\infty)$ be functions such that
\begin{enumerate} [(i)]
\item $L$ attains its minimum on sets of the form
\[
\{ \bX \in \Lambda \setminus \{ \bzero \}: N(\bX) \leq B \}  \textrm{ for any } B\in \bR_{\ge 0},
\]
\item
\begin{equation} \label{nonzero}
L(\bX) \ne \bzero \qquad (\bzero \ne \bX \in \Lam),
\end{equation}
\item
and
\begin{equation} \label{infimum}
\inf_{\bX \in\Lambda \setminus \{\bzero \}} L(\bX) = 0.
\end{equation}
\end{enumerate}
A \emph{sequence of $(N,L)$-best approximations} is $(\bX_i)_{i=1}^\infty \in \Lambda^\bN$ such that
\begin{enumerate}[(i)]
\item{the sequence $N(\bX_1), N(\bX_2),\ldots$ is strictly increasing;}
\item{the sequence $L(\bX_1), L(\bX_2), \ldots $ is strictly decreasing;}
\item{for every $\bX \in \Lambda \setminus \{\bzero\}$,
\begin{equation}\label{defBA} \textrm{if } N(\bX) < N(\bX_{i+1}) \textrm{ then } L(\bX) \geq L(\bX_i). \end{equation}}
\end{enumerate}
\end{defn}

\noindent
Informally, the function $N$ (usually a height) measures the size of a point in $\Lambda$. The function $L$ will, in practice, depend on some point in $E$ that we wish to approximate; it measures the quantitative strength of the best approximations. The condition $\displaystyle \inf_{\bX \in\Lambda \setminus \{\bzero\} } L(\bX) =0$ ensures that good approximations exist at all.

In the context of weighted exponents, given $A \in M_{m \times n}(\bR)$ we shall consider
\begin{equation}
\label{LENL}
\Lambda = \bZ^m, \qquad N(\bX) = \|\bX\|_{\bs} , \qquad L(\bX) = \displaystyle \inf_{\bp \in \bZ^n} \|{}^tA\bX - \bp \|_\br,
\end{equation}
and call $\bX_1, \bX_2, \ldots$ \emph{weighted best approximations} for $A$. We will require an additional hypothesis on $A$ to ensure that
\[
\{ ^t A \bX: \bX \in \bZ^m \setminus \{ \bzero \}, \bp \in \bZ^n \}
\]
does not contain the origin, but contains points arbitrarily close to it.

Observe that the requirement (i) on $N$ and $L$ is met if $N$ possesses the \emph{Northcott property}
\[
\#\{\bX \in \Lambda: N(\bX) \leq B\} < \infty \textrm{ for any } B \in \bR_{\ge 0}.
\]
For example, one may consider best approximation vectors and exponents of best approximation for $\Lambda$ being the ring of integers of a Northcott field.

For $i \in \bN$ we write $Y_i=N(\bX_i)$ and $M_i=L(\bX_i)$. The sequences $(Y_i)_{i=1}^\infty$ and $(M_i)_{i=1}^\infty$ are respectively strictly increasing and strictly decreasing, and furthermore
\[
\lim_{i\rightarrow \infty} Y_i=\infty, \qquad \lim_{i\rightarrow \infty} M_i=0.
\]
We will see from the construction below that there can be several distinct sequences of $(N,L)$-best approximations. Notwithstanding, the sequences $(Y_i)_{i>0}$ and $(M_i)_{i>0}$ are uniquely determined.\\

We now demonstrate, by construction, the existence of a sequence of $(N,L)$-best approximations. When $i=0$, let $M_0$ be the minimum of $L$ on the set of $\bX \in \Lambda \setminus \{ \bzero \}$ such that $N(\bX) \leq 1$, and choose $\bX_0$ to be a point were this minimum in reached. Neither $\bX_0$ nor $Y_0 := N(\bX_0)$ is uniquely determined. Next, suppose that $\bX_1, \ldots, \bX_k$ have already been chosen, that \eqref{defBA} holds for $i \le k-1$, and that
\[
 \text{if } \bX \in \Lambda \setminus \{\bzero\} \text{ and } N(\bX) < N(\bX_{k}) \text{ then }L(\bX) \ge L(\bX_k).
\]
Let $Y > N(\bX_{k})$ be minimal such that
\[
\min \{ L(\bX): \bX \in \Lambda, \: N(\bX) \leq Y \} < M_k.
\]
This minimum is well-defined, since $\displaystyle \inf_{\bX \in\Lambda} L(\bX) = 0$ and $L$ is  strictly positive. Putting $Y_{k+1}=Y$, by the definition of $Y$ there exists a (non-unique) point $\bX_{k+1}$ such that $Y_{k+1} = N(\bX_{k+1})$ and $L(\bX_{k+1}) = M_{k+1} < M_k$ satisfy \eqref{defBA}. From our recursive definition, it is clear that the sequence $\bX_1, \bX_2, \ldots$ so constructed is a sequence of $(N,L)$-best approximations.\\

For a sequence of $(N,L)$-best approximations, define the exponents
\[
\omega_{N,L} := \limsup_{i\to \infty} \cfrac{\log(M_i)}{\log(Y_i)}, \qquad
\hat{\omega}_{N,L} :=  \liminf_{i\to \infty} \cfrac{\log(M_i)}{\log(Y_{i+1})}.
\]
The quantities are well-defined, and compatible with our previous definitions.

\begin{lemma} \label{suprema}
In the setting \eqref{LENL}, we have
\[
\omega_{\br,\bs}({}^tA) = {\omega}_{N,L},
\qquad
\hat{\omega}_{\br,\bs}({}^tA) = \hat{\omega}_{N,L}.
\]
\end{lemma}

\begin{proof}
From the definitions, if $M_i\le Y_i^{-\omega}$ for infinitely many indices $i$, then $\omega\le \omega_{\br,\bs}({}^tA)$. This shows that $\omega_{\br, \bs}({}^tA) \ge {\omega}_{N,L}$.

Conversely, for any $\omega < \omega_{\br,\bs}({}^tA)$, we can choose $T > 0$ arbitrarily large such that
\[\|{}^tA\bp-\bq\|_{\br}<T^{-\omega}, \qquad \|\bp\|_{\bs}< T\]
for some $(\bp, \bq)\in (\bZ^m\setminus\{\bzero\})\times\bZ^n$. Let $k$ be the index such that $Y_k\le T< Y_{k+1}$. From the definition of weighted best approximations, we have
\[
M_k\le \|{}^tA\bp-\bq\|_{\br}<T^{-\omega} \le Y_k^{-\omega}.
\]
We conclude that ${\omega}_{N,L} \ge \omega_{\br, \bs}({}^tA)$.

We have proved the first assertion, and the second follows by a similar argument.
\end{proof}

Next we show, under a further assumption on $L$ and $N$, that the sequence $(Y_i)_{i=1}^\infty$ exhibits geometric growth, generalising the fact that continued fraction denominators enjoy this property. This is well-understood in the context of unweighted best rational approximations, see \cite[Lemma 1]{BL2005} and \cite{Lag1, LagAus}.

\begin{lemma}\label{lem-geo-grow}
Suppose that $L$ and $N$ satisfy the slack triangle inequalities 
\begin{align} \notag
 N(\lambda(\ba +\bb)) &\leq \lambda^{\delta}( N(\ba) + N(\bb) ) \qquad (\lam \le 1) \\
\label{NLgeom}
L(U^{-1}(\ba +\bb)) &\leq \frac12 (L(\ba) +L(\bb))
\end{align}
for some $\delta>0$ and all $U \ge U_0(\del)$. Then there exist $c>0$ and $\gamma>1$ such that
  \[Y_i\ge c\gamma^i \qquad (i \in \bN).\]
\end{lemma}

\begin{remark}\label{rem-delta}
Note that the inequalities \eqref{NLgeom} hold in the setting \eqref{LENL} of weighted exponents, with $\delta=\min\{\delta_\bs, \delta_{\br}\}=\min\{s_i,r_j: 1\le i\le m, \: 1\le j \le n\}$. % 
\iffalse
However, it fails to hold in the context of multiplicative exponents, wherein the natural analogue would be $N(\bX) = \Pi_+(\bX)$ and $L(\bX) = \displaystyle \inf_{\bp \in \bZ^m} \Pi (A\bX - \bp)$. This explains why we do not have equality for almost all $\btet$ for the multiplicative version of Theorem \ref{MainThm}.
\fi
\end{remark}

\begin{proof}
Let us take $U\in \bN$ constant, but large enough to ensure that
\begin{equation}\label{equ-delta}
  U^{\delta}>3. %, \text{ where } \delta=\min\{s_i,r_j: 1\le i\le m, 1\le j \le n\}.
\end{equation}
Consider any $V :=  2U^{n} $ consecutive vectors $\bX_{i+1},\ldots, \bX_{i+ V}$. By the pigeonhole principle, there exist indices $j,k$ with $1 \le j<k\le V$ such that $\bX_{i+j}-\bX_{i+k}\in U\bZ^n$. By the slack triangle inequality, the vector
\[
\bX= U^{-1} (\bX_{i+k}-\bX_{i+j})
\]
satisfies
\[
N(\bX)= N( U^{-1} (\bX_{i+k}-\bX_{i+j}) ) \le U^{-\del} (Y_{i+j}+Y_{i+k})\le \frac{1}{3} (Y_{i+j}+Y_{i+k})
\]
and
\[
L(\bX) = L(U^{-1} (\bX_{i+k}-\bX_{i+j})) \le \frac{1}{3} (M_{i+j}+M_{i+k}) < M_{i+j}.  \]
By the definition of $(N,L)$-best approximations, we must have
\[ Y_{i+j+1}\le N(\bX) \le \frac{1}{3} (Y_{i+j}+Y_{i+k}), \]
and consequently
\[
Y_{i+V}\ge Y_{i+k} \ge 3Y_{i+j+1}-Y_{i+j}\ge 2Y_{i+j+1} \ge 2Y_i
\qquad (i \in \bN).
\]
With $\gam = 2^{1/V}$, the inequality $Y_i \gg \gam^i$ now follows by induction.
\end{proof}

\iffalse
\[\bq_{k+1}=\]
Then we choose $\bq_{k+1}$ such that  is smallest among all
\[\|\bq_{k+1}\|_{\br}=\inf\{\|\bq\|_{\br}: \|\bq\|_{\br}>\|\bq_{k}\|_{\br}, \langle A\bq\rangle_{\bs}< \langle %A\bq_k\rangle_{\bs}\} \]
\fi

\section{The Bugeaud--Laurent theorem with weights}
\label{S5}

Our objective in this section is to establish Theorem \ref{MainThm}. First and foremost, we use Lemma \ref{lem-psi} to deduce \eqref{equ-main-ine}. By definition, if $\hat{\omega}>\hat{\omega}_{\bs,\br}(A)$ then $A$ is not uniformly $(\hat{\omega},\bs,\br)$-approximable. Observe that the functions $\phi(T)=T^{-\omega}$ and $\psi(T)=T^{-\hat{\omega}}$ satisfy the condition \eqref{equ-psi-phi} whenever $\omega\hat{\omega}>1$. Therefore $(^tA,\btet)$ is $(\omega, \br, \bs)$-approximable for all $\btet\in \bR^n$, by Lemma \ref{lem-psi}. Since $\omega, \hat{\omega}$ are arbitrary real numbers for which $\omega_{\br,\bs}(^tA,\btet) \ge \omega > \hat \omega^{-1}$, this confirms the second inequality in \eqref{equ-main-ine}. The proof of the first inequality is similar.\\

Next, we show that for any fixed $A$, equality holds in the first inequality of \eqref{equ-main-ine} for almost all $\btet\in \bR^n$; the analogous statement for the second inequality will follow by similar reasoning. We will apply the theory of weighted best approximations in the setting \eqref{LENL}, when $A \in M_{m \times n} (\bR)$ has the property that $G := {^t A} \bZ^m + \bZ^n \le \bR^m$ has maximal rank $m+n$ as a group. This implies \eqref{nonzero} and, by Kronecker's theorem \cite[Chapter III, Theorem IV]{Cas1}, also ensures that the condition \eqref{infimum} is met.

When $G$ does not have maximal rank, the observation made by Bugeaud and Laurent \cite{BL2005} still applies in our weighted framework: the exponents $\omega_{\br, \bs}(^t A, \btet)$ and $\hat \ome _{\br, \bs} (^t A, \btet)$ vanish unless $\btet$ lies in a discrete family of parallel hyperplanes in $\bR^n$. The upshot is that, in this case, equality in \eqref{equ-main-ine} certainly holds almost surely.\\

We proceed on the assumption that $G$ has maximal rank. It suffices to prove that if $0 < \eps < \hat \omega_{\bs,\br}(A)$ then the set of $\btet \in [0,1]^n$ satisfying
 \begin{equation}\label{EpsIneq}
\omega_{\br,\bs}(^tA,\btet)> \frac{1+\eps}{\hat{\omega}_{\bs,\br}(A)-\eps}+\eps
 \end{equation}
\iffalse
\quad \text{ or } \quad \hat{\omega}_{\br,\bs}(^tA,\btet)\ge \frac{1}{\omega_{\bs,\br}(A)-\eps}+\eps
\fi
has Lebesgue measure $0$. Denote by $\langle\cdot,\cdot \rangle$ the inner product on $\bR^n$, that is, write
 \[\langle \bx, \by\rangle =x_1y_1+ \cdots+x_ny_n.\]
We compute that
 \[
|\langle \bx, \by\rangle|\le \sum_{i=1}^{n}|x_iy_i|\le \sum_{i=1}^{n}\|\bx\|_{\br}^{r_i}\|\by\|_{\br}^{r_i} \le n \max \{ \|\bx\|_{\br}^{\delta}  \|\by\|_{\br}^{\delta},  \| \bx\|_{\br} \|\by\|_{\br} \},
\]
 where $\delta$ is as in Remark \ref{rem-delta}.
\iffalse
Consequently,
\[\|\bx\|_{\bs}\|\by\|_{\bs}\ge \min\left\{\frac{1}{n}|\langle \bx, \by\rangle|, 1\right\}^{\frac{1}{\delta}}.\]
\fi

Fix a positive real number $\eps < \hat \omega_{\bs,\br}(A)$, as well as a sequence $\bq_1, \bq_2, \ldots$ of weighted best approximations for ${}^tA$, and for $k \in \bN$ write $Y_k=\| \bq_k\|_{\br}$ as before. We claim that for $\eta=\frac{1}{2}\delta\eps >0$, the set of $\btet\in [0,1]^n$ satisfying \eqref{EpsIneq} is contained in $\displaystyle \limsup_{k \to \infty} S_k$, where
 \begin{equation*}
  S_k=\{\by\in [0,1]^n: \dist(\langle \by, \bq_k\rangle, \bZ)<Y_k^{-\eta}\} \qquad (n \in \bN).
 \end{equation*}
It is easy to check that the Lebesgue measure of $S_k$ does not exceed $2n Y_k^{-\eta}$. Thus, by Lemma \ref{lem-geo-grow} and the first Borel--Cantelli lemma, the theorem will follow from our claim. It remains to confirm the claim.

Suppose $\btet\in [0,1]^n$ satisfies \eqref{EpsIneq}, and put
 \[
\hat{\omega}=\hat{\omega}_{\bs,\br}(A)-\eps, \qquad
\omega=\frac{1+\eps }{\hat{\omega}}+\eps.
\]
By \eqref{EpsIneq} we have $\omega < \omega_{\br, \bs}(^t A, \btet)$, so we may select an arbitrarily large positive real number $T$ such that the inequalities
 \[\|\bp\|_{\bs}<T, \qquad \|^tA\bp-\bq-\btet\|_{\bs}<T^{-\omega}\]
 have a  solution $(\bp,\bq)\in \bZ^m\times(\bZ^n\setminus \{\bzero\})$.
 Let $k$ be the unique index for which $Y_k\le T^{\omega-\eps}< Y_{k+1}$, so that Lemma \ref{suprema} gives
 \[M_k<Y_{k+1}^{-\hat{\omega}}<T^{-\hat{\omega}(\omega-\epsilon)}=T^{-1-\eps}.\]
Since $A$ is the Hermitian adjoint of $^tA$, we obtain
\begin{align}
  \langle\bq_k, \btet\rangle &= \langle\bq_k, \prescript{t}{}{A}\bp\rangle + \langle\bq_k, \bq\rangle -\langle\bq_k, \prescript{t}{}{A}\bp-\bq-\btet\rangle \notag \\
  &\equiv \langle A\bq_k -\bp_k, \bp\rangle-\langle\bq_k, \prescript{t}{}{A}\bp-\bq-\btet\rangle \qquad \mmod 1, \label{equ-basic}
\end{align}
where $\bp_k$ is the integer vector nearest to $A \bq_k$. Moreover, as $T$ is large we have
\[
|\langle A\bq_k-\bp_k, \bp\rangle| \le m \cdot \max\{M_k^{\delta}\|\bp\|_{\bs}^{\delta}, M_k\|\bp\|_{\bs}\} \le m T^{-\delta\eps }<\frac{1}{2}T^{-\eta}
\]
and
\[
|\langle\bq_k, \prescript{t}{}{A}\bp-\bq-\btet\rangle|\le n \cdot \max\{ Y_k^{\delta}T^{-\omega\delta}, Y_kT^{-\omega} \}\le n T^{-\delta\eps}<\frac{1}{2}T^{-\eta}.
\]
The triangle inequality now reveals that $\btet \in \displaystyle \limsup_{k \to \infty} S_k$, establishing our claim and hence the theorem.

\section{Weighted inhomogeneous Bad}\label{S5-1}
In this section, we establish Theorem \ref{thm-bklr}, closely following \cite{BKLR}. We begin by recalling the setup. For $\eps > 0,$ set
\[
\Bad_{\br,\bs}^{\eps}({}^tA) = \{ \btet\in \bR^n: \liminf_{(\bp,\bq)\in (\bZ^{m}\setminus\{\bzero\})\times\bZ^n} \|\bp\|_{\bs}\|{}^tA\bp-\bq-\btet\|_{\br}\ge \eps \},
\]
and 
\[\Bad_{\br,\bs}({}^tA) = \bigcup_{\eps > 0} \Bad_{\br,\bs}^{\eps}({}^tA). \]

It is known that $\Bad_{\br,\bs}({}^tA)$ has full Hausdorff dimension, see \cite{KW2010}, however the task of determining the Hausdorff dimension of $\Bad_{\br,\bs}^{\eps}({}^tA)$ is much more delicate. The case of vectors in $\bR^n$ has been recently studied in \cite{LSS}, where it is proved that for $\eps > 0$ and for an explicit class of vectors $\mathbf{v}$ termed \emph{heavy}, we have
\begin{equation}\label{heavy}
\dim(\Bad_{\br,\bs}^{\eps}(\mathbf{v})) < n.
\end{equation}
Heavy vectors form a set of full Lebesgue measure. 

Subsequently, the unweighted sets $\Bad^{\eps}(A)$ were investigated in \cite{BKLR}, where necessary and sufficient conditions were obtained in dimension $1$, so that $\dim(\Bad^{\eps}(v)) = 1$ for some $\eps > 0$. These conditions were expressed in terms of the continued fraction expansion of $v$, and were shown to be equivalent to $v$ being ``singular on average''. In higher dimensions, \cite[Theorem 1.5]{BKLR} states sufficient conditions, in terms of best approximation vectors for a matrix $A$, to ensure that $\dim(\Bad^{\eps}(A)) = n$ for some $\eps > 0$. Theorem \ref{thm-bklr} is a weighted extension of the latter result. 

We begin by proving  a weighted version of  \cite[Theorem 5.1]{BKLR}.

\begin{thm}\label{thm-bklr-w}
For any $\alpha\in (0,1/2)$, there exists $R=R(\alpha)>1$ with the following property. Let $(\by_k)_{k\ge 1}$ be a sequence in $\bR^n \setminus \{ \bzero \}$ such that $\|\by_{k+1}\|_{\br}/\|\by_{k}\|_{\br}\ge R$ for all $k\ge 1$ and $\lim_{k\rightarrow \infty}\|\by_k\|_{\br}^{1/k}=\infty$. Then the set
  \[S_{\alpha} :=\{\btet\in [0,1]^n: \dist(\langle\by_k,\btet\rangle, \bZ)\ge\alpha \text{ for all } k\ge 1  \}\]
  has Hausdorff dimension $n$.
\end{thm}

\begin{proof}
Let $\alpha\in (0,1/2)$ be fixed. Choose $R$ large enough that
\[c:=1-2\alpha-3nR^{-\delta_\br}>0.\]
We proceed to verify the theorem for such $R$. By the mass distribution principle \cite[Chapter 4]{Falconer}, it suffices to demonstrate the existence of a measure $\mu$, supported on $S_\alpha$, with the following property:  if $\eps > 0$ then there exist $C(\eps), r_0(\eps) >0$  such that for any Euclidean ball $B$ of radius $r \in (0, r_0(\eps))$ we have
\begin{equation} \label{mtp}
\mu(B)\le C(\eps)r^{n-\eps}.
\end{equation}
The measure $\mu$ will be constructed in a standard way. Write
\[
Y_k=\|\by_k\|_{\br}, \qquad
Z_{k, \alpha}=\{\btet\in [0,1]^n: \dist(\langle\by_k,\btet\rangle, \bZ) < \alpha  \}.
\]

The first step is to construct two sequences $(\cI_k)_{k\ge 1}$ and $(\cJ_k)_{k\ge 1}$ of collections of subsets of $I_0 := [0,1]^n$. Set $\cI_0= \{I_0\}$. The sequences $(\cI_k)_{k\ge 1}$ and $(\cJ_k)_{k\ge 1}$ will be defined recursively, so that $\cJ_k\subset \cI_k$, and so that $\cI_k$ comprises a collection of translates of $\Pi(Y_k)$, where \[\Pi(Y_k)=[0,Y_k^{-r_1}]\times\cdots\times [0,Y_k^{-r_n}].\]
  It is easily seen that there exists a collection $\cP_{k+1}$ of translates of $\Pi(Y_{k+1})$ satisfying:
  \begin{enumerate}
    \item Elements from $\cP_{k+1}$ are subsets of $\Pi(Y_{k})$ and have mutually disjoint interiors;
    \item \[\Delta_{k+1}:=\#\cP_{k+1} =\prod_{1\le i\le n}\lfloor Y_{k}^{-r_i}Y_{k+1}^{r_i}\rfloor. \]
  \end{enumerate}
We record the following lower bound on $\Delta_{k+1}$ for later use:
\begin{align*}
  \Delta_{k+1} &\ge \prod_{1\le i\le n}(Y_{k}^{-r_i}Y_{k+1}^{r_i}-1)\ge Y_k^{-1}Y_{k+1}(1-\sum_{1\le i\le n} Y_k^{r_i}Y_{k+1}^{-r_i}) \\&\ge (1-nR^{-\delta_\br})Y_k^{-1}Y_{k+1}.
\end{align*}
 Now assume that $\cI_{k}$ and $\cJ_k$ have been defined. For any $I=\btet+ \Pi(Y_k)\in \cI_{k}$, set
 \[\cI(I)=\{\btet+I_1: I_1\in \cP_{k+1}\}  \text{ and } \cJ(I)=\{I_2\in \cI(I): I_2\cap Z_{k,\alpha}=\emptyset\}.\]
Then choose
 \[\cI_{k+1}=\bigcup_{I\in \cI_{k}} \cI(I) \text{ and } \cJ_{k+1}=\bigcup_{I\in \cJ_{k}} \cJ(I).\]

 For any $I\in \cI_{k}$, we have $\# \cI(I)=\Delta_{k+1}$. Next, we estimate $\# \cJ(I)$. Observe that if $I_3\in \cI(I)$, with $I_3\cap Z_{k,\alpha}\ne \emptyset$, then $I_3\subset Z_{k,\beta}$, where
 \[\beta=\alpha+\sum_{1\le i\le n} Y_k^{r_i}Y_{k+1}^{-r_i}\le \alpha+nR^{-\delta_\br}.\]
Hence
 \[\# \cJ(I)\ge \Delta_{k+1}- \frac{\lambda(Z_{k,\beta}\cap I)}{\lambda(I_3)} \ge \Delta_{k+1}-2\beta Y_k^{-1}Y_{k+1}\ge cY_k^{-1}Y_{k+1},\]
 where $\lambda$ denotes Lebesgue measure on $\bR^n$. Consequently, we have the lower bound
 \[\# \cJ_{k}\ge c^{k-1}Y_1^{-1}Y_{k}.\]

We are ready to specify the measure $\mu$, as promised. Define
 \[
\mu=\lim_{k \to \infty} \mu_k,
\]
where
\[
\mu_k= (\# \cJ_{k})^{-1}Y_k \sum_{I\in \cJ_k} \lambda|_{I}.
\]
 %\mu_k=(\sum_{I\in \cJ_k}\lambda(I))^{-1}\sum_{I\in \cJ_k} \lambda|_{I}
It follows from the construction that $\mu$ is a probability measure supported on $S_{\alpha}$. Let $B$ be a ball of radius $r \in (0, r_0(\eps))$. Choose $k$ such that 
\[
Y_{k+1}^{-\delta_{\br}}< r\le Y_{k}^{-\delta_{\br}}.
\]
Then $B$ can be covered by at most $4^nr^nY_{k+1}$ many elements from $\cI_{k+1}$. Thus
 \[\mu_{k+1}(B)\le 4^nr^nY_{k+1}(\# \cJ_{k+1})^{-1}\le 4^nr^nc^kY_1\le C_1r^{n - \frac{k\log c}{\delta_{\br}\log Y_k}},\]
 for some $C_1>0$. By our assumption that $Y_k^{1/k} \to \infty$, this implies \eqref{mtp}, which completes the proof.
\end{proof}

We are equipped to prove Theorem \ref{thm-bklr}. Fix $\alpha\in (0,1/2)$, and let $R=R(\alpha)$ be as in Theorem \ref{thm-bklr-w}. For $k \in \bN$, write 
\[
Y_k=\|\bq_k\|_{\br}, \qquad M_k = \inf_{\bp \in \bZ^m} \| A \bq_k - \bp \|_{\bs}
= \| A \bq_k - \bp_k \|_\bs,
\]
as before. Since $\displaystyle \lim_{k \rightarrow \infty} Y_k^{1/k}=\infty$, the proof of \cite[Theorem 2.2]{BKLR} reveals that there exists a function $\phi:\bN \rightarrow \bN$ for which
\[Y_{\phi(k+1)}\ge RY_{\phi(k)} \quad \text{ and } \quad Y_{\phi(k)+1}\ge R^{-1}Y_{\phi(k+1)}.\]
%\[\|\bq_{\phi(i)}\|_{\br}\ge R\|\bq_{\phi(i-1)}\|_{\br} \quad \text{ and } \quad \|\bq_{\phi(i-1)+1}\|_{\br}\ge %R^{-1}\|\bq_{\phi(i)}\|_{\br}.\]
By Theorem \ref{thm-bklr-w}, it suffices to show that the set
 \[S :=\{\btet\in [0,1]^n: \dist(\langle\bq_{\phi(k)},\btet\rangle, \bZ)\ge\alpha \text{ for all } k\ge 1  \}\]
 is a subset of $\Bad_{\br,\bs}^{\eps }({}^tA)$, for some $\eps >0$.

Let $\btet\in S$ . For any $(\bp,\bq)\in (\bZ^{m}\setminus\{\bzero\})\times \bZ^n$, let $k$ be the unique index for which
\[
Y_{\phi(k)}\le \eps_1^{-1}\|\bp\|_{\bs}<Y_{\phi(k+1)},
\]
where $\eps_1=R^{-1}(\alpha/2m)^{1/\delta}$. With $\del$ as in Remark \ref{rem-delta}, we have
\begin{align*}
  |\langle A\bq_{\phi(k)}-\bp_{\phi(k)}, \bp\rangle| &\le m \cdot \max\{M_{\phi(k)}^{\delta}\|\bp\|_{\bs}^{\delta}, M_{\phi(k)}\|\bp\|_{\bs}\} \\
  &\le m \cdot \max\{Y_{\phi(k)+1}^{-\delta}\|\bp\|_{\bs}^{\delta}, Y_{\phi(k)+1}^{-1} \|\bp\|_{\bs}\} \le m (R\eps_1)^{\delta}\le \frac{\alpha}{2}.
\end{align*}
Now the calculation \eqref{equ-basic} yields, for any $(\bp,\bq)\in (\bZ^{m}\setminus\{\bzero\})\times\bZ^n$, the inequality
\[ |\langle\bq_{\phi(k)}, {}^t{A}\bp-\bq-\btet\rangle| \ge \frac{\alpha}{2}.\]
Hence
 \[Y_{\phi(k)}\|{}^t{A}\bp-\bq-\btet\|_{\br}\ge \min\left\{\frac{\alpha}{2n}, \left(\frac{\alpha}{2n}\right)^{1/\delta}\right\}=\left(\frac{\alpha}{2n}\right)^{1/\delta},\]
 which implies
 \[\| \bp \|_{\bs} \cdot \|{}^t{A}\bp-\bq-\btet\|_{\br} \ge \eps_1 \left(\frac{\alpha}{2n}\right)^{1/\delta}= \frac{1}{R} \left(\frac{\alpha^2}{4mn}\right)^{1/\delta}. \]
Therefore $\btet \in \Bad_{\br, \bs}^\eps(^t A)$, for $\eps =  \frac{1}{R} \left(\frac{\alpha^2}{4mn} \right)^{1/\del}$, completing the proof of Theorem \ref{thm-bklr}.

\section[Inhomogeneous DA on manifolds]{Applications to inhomogeneous Diophantine approximation on manifolds} \label{S2}

Transference principles such as Theorem \ref{MainThm} play a crucial role in inhomogeneous Diophantine approximation on manifolds by providing \emph{lower bounds} for inhomogeneous Diophantine exponents. In this section, we illustrate this principle in a number of examples. The corresponding upper bound in each case is found using a (suitable adaptation of) the transference principle of Beresnevich and Velani. We begin with

 \begin{prop}
Assume that $\omega_{\bs, \br}(A) = 1$. Then for every $\btet \in \bR^m$, $$\omega_{\bs, \br}(A, \btet) \geq 1,$$
with equality for Lebesgue-almost all $\btet$.
\end{prop}

\begin{proof}

We follow the argument in \cite{BV2010}. Let $A \in M_{m\times n}(\bR)$ be as above. By the weighted version of Dyson's transference principle (Theorem \ref{thm-khin}), we have $\omega_{\br, \bs}({}^tA) = 1$. Now applying the weighted form of Dirichlet's approximation theorem and the trivial inequality $\omega_{\br, \bs}(^t A) \geq  \hat{\omega}_{\br, \bs}({}^t A)$ yields
\[
1 = \omega_{\br, \bs}({}^tA) \ge \hat{\omega}_{\br, \bs}({}^t A) \ge 1.
\]
These inequalities must be equalities, so in particular $\hat{\omega}_{\br, \bs}({}^t A) = 1$. By Theorem \ref{MainThm}, we finally have $\omega_{\bs, \br}(A,\btet) \geq 1$, with equality for almost all $\btet$.
\end{proof}

As a consequence we have:
 \begin{cor}\label{cor:upper}
Let $\mu$ be a measure on $M_{m\times n}(\bR)$. Assume that $\omega_{\bs, \br}(A) = 1$ for $\mu$-almost all $A$. Then for every $\btet \in \bR^m$ and $\mu$-almost all $A$, we have $$\omega_{\bs, \br}(A, \btet) \geq 1.$$
\end{cor}

We now discuss the above corollary in the context of some interesting measures. Dirichlet's theorem implies that $\omega(A) \geq 1$ for every $A \in M_{m\times n}(\bR)$. A matrix $A \in M_{m\times n}(\bR)$ is said to be \emph{very well approximable} if $\omega(A) > 1$.  Metric Diophantine approximation on manifolds is concerned with the question of whether typical Diophantine properties in $M_{m\times n}(\bR)$, i.e. those which are generic for Lebesgue measure, are inherited by proper submanifolds (or, more generally, supports of suitable measures).

In $1932$, Mahler \cite{Mah1932} conjectured that for almost every $x \in \bR$ the vector
\begin{equation}\label{def:Mahler}
(x, x^2, \dots, x^n)
\end{equation}
is not very well approximable. A measure $\mu$ is called \emph{extremal} if $\mu$-almost every $A$ is not very well approximable. A manifold is \emph{extremal} if a measure in its natural volume class---for example, the pushforward of Lebesgue measure by a  map parametrising the manifold---is extremal. We can similarly define (weighted) \emph{inhomogeneously very well approximable} matrices (see \cite[\S\S 6.1--6.2]{BKM15}) and (inhomogeneously) \emph{very well multiplicatively approximable} matrices, as well as the corresponding notions of extremality. \\

Mahler's conjecture was resolved by Sprind\v{z}uk \cite{Sp1, Sp2}, who in turn formulated a more general conjecture \cite{Sp3} that was proved by Kleinbock and Margulis \cite{KM}. They showed that almost every point on a smooth, ``nondegenerate" submanifold of $\bR^n$ is not very well (multiplicatively) approximable. Subsequently, there have been numerous advances in the subject. The inhomogeneous version of Sprind\v{z}uk's conjectures were established by Beresnevich and Velani \cite{BV2010} using their \emph{transference principle}. They proved, for instance, the following result  (\cite[Theorem 1]{BV2010}).
\begin{thm}
Let $\mu$ be a measure on $M_{m\times n}(\bR)$. If $\mu$ is contracting almost everywhere, then for every $\btet \in \bR^m$, we have that
$$\omega(A, \btet) = 1$$
for $\mu$-almost every $A \in M_{m\times n}(\bR)$.
\end{thm}

We refer the interested reader to the papers above for the definitions of the terms ``nondenegenerate" and ``contracting". Since there is no inhomogeneous version of Dirichlet's theorem, the notion of extremality is more delicate and both upper and lower bounds for the exponent are required. The proof of the Theorem above accordingly has two steps. The lower bound for the exponent proceeds using the transference results of Bugeaud and Laurent, combined with Dyson's transference principle. The upper bound is proved using a homogeneous--inhomogeneous transference principle introduced for the purpose, by Beresnevich and Velani.

Subsequently, Beresnevich, Kleinbock and Margulis \cite{BKM15} considered the more general problem of Diophantine approximation on manifolds in the space of matrices. We recall their notation. Let $X$ be a Euclidean space. Given $x \in X$ and $r > 0$, let $B(x,r)$ denote the open ball of radius $r$ centred at $x$. If $V = B(x,r)$ and $c > 0$, let $cV$ stand for $B(x,cr)$. Let $\mu$ be a Radon measure on $X$. Given $V \subset X$ such that $\mu(V) > 0$ and a function $f:V \to \bR$, let
$$\|f\|_{\mu, V} := \sup_{x \in V \cap \supp~\mu} |f(x)|. $$
A Radon measure $\mu$ will be called \emph{$D$-Federer} on $U$, where $D > 0$ and $U$ is an open subset of $X$, if $\mu(3V) < D\mu(V)$ for any ball $V \subset U$ centred in the support of $\mu$. The measure $\mu$ is called \emph{Federer} if for $\mu$-almost every point $x \in X$ there is a neighbourhood $U$ of $x$ and $D > 0$ such that $\mu$ is $D$-Federer on $U$.

Given $C, \alpha > 0$ and an open subset $U \subset X$, we say that $f : U \to \bR$ is called \emph{$(C, \alpha)$-good} on $U$ with respect to the measure $\mu$ if for any ball $V \subset U$ centred in $\supp \: \mu$ and any $\varepsilon > 0$ one
has
$$ \mu(\{x \in V ~:~|f(x)| < \varepsilon\}) \leq C\left(\frac{\varepsilon}{\|f\|_{\mu, V}} \right)^{\alpha} \mu(V).$$
Given $f = (f_1, \dots, f_N) : U \to \bR^N$, we say that the pair $(f, \mu)$ is \emph{good} if for $\mu$-almost every $x \in U$ there is a neighbourhood $V \subset U$ of $x$ and $C, \alpha > 0$ such that any linear combination of $1, f_1, \dots, f_N$ over $\bR$ is $(C, \alpha)$-good on $V$. The pair $(f, \mu)$ is called \emph{non-planar} if
for any ball $V \subset U$ centred in $\supp~\mu$ the set $f(V \cap \supp~\mu)$ is not contained in any affine hyperplane of $\bR^N$. Non-planarity is a generalisation of the nondegeneracy property of smooth manifolds mentioned above. The theorem below deals with Diophantine properties of submanifolds in the space of matrices, and therefore requires a slightly more general notion. Informally, we say that the pair $(F, \mu)$ is weakly non-planar if $F(\supp \: \mu)$ does not locally lie entirely inside a certain polynomial hypersurface. The precise definition may be found in \cite[\S 2]{BKM15}.

The following is \cite[Theorem 6.5]{BKM15}.

\begin{thm} \label{thm:BKM}
 Let $U$ be an open subset of $\bR^d$, $\mu$ a Federer measure on $U$, and $F : U \to M_{m \times n}(\bR)$ a continuous map such that $(F, \mu)$ is
 \begin{enumerate}
 \item good, and
 \item weakly non-planar.
 \end{enumerate}
 Then $F_{*}\mu$ is inhomogeneously $(\bs, \br)$-extremal for any $(m + n)$-tuple $(\bs, \br)$ of weights.
\end{thm}

The conclusion of the theorem above asserts that $\omega_{\bs, \br}(A,\btet) \leq 1$ for $\mu$-almost every $A$ in $F(U)$. We may apply Corollary \ref{cor:upper} using the homogeneous ($\btet = \bzero$) case of Theorem \ref{thm:BKM}, along with the weighted version of Dirichlet's theorem, to conclude thus.

 \begin{thm}
 Let $U, \mu, F$ be as in Theorem \ref{thm:BKM}. Then for $\mu$-almost every $A \in F(U)$ and every $\btet$, we have
 $$\omega_{\bs, \br}(A,\btet) = 1.$$
 \end{thm}

\noindent Whilst the scope of the result is very general, we note that it is new even for the Veronese curve (\ref{def:Mahler}) in both the simultaneous and the dual approximation contexts.\\

We now turn our attention to measures which do not satisfy the weak non-planarity condition.  Natural examples are provided by affine subspaces, and the reader may consult \cite{Ghosh} for a recent survey. We specialise to the case of simultaneous approximation. The results below also hold in the dual setting: see \cite{BGGV} for examples of inhomogeneous dual Diophantine approximation on affine subspaces.

\begin{thm}\label{thm:K1}
Let $L$ be an $\br$-extremal affine subspace of $\bR^n$ and let $M$ be a smooth nondegenerate submanifold of $L$. Then, for almost every $A \in M$ and every $\btet \in \bR^n$,
$$\omega_{\br}(A,\btet) = 1.$$
\end{thm}

\begin{proof}

The unweighted analogue of this result was proved in \cite{BV2010}. First we note that $L$ is $\br$-extremal if and only if $M$ is. This is more or less proved by Kleinbock in \cite{Kleinbock}. More precisely, the unweighted (Theorem $1.2$) and multiplicative (Theorem $1.4$) variants of the statement are known. We indicate the minor changes required to prove the weighted version. The argument in \cite{Kleinbock} is based on the dynamical approach developed in \cite{KM}. To a row vector $\by \in \bR^n$, one can associate a unimodular lattice
$$u_\by \bZ^{n+1} := \begin{pmatrix}1 &y\\0 &I_n \end{pmatrix}.$$

The Diophantine properties of $\by$ may be viewed in terms of the $F$-orbit of $u_\by \bZ^{n+1}$ on the moduli space of lattices $\SL_{n+1}(\bR)/\SL_{n+1}(\bZ)$,  where $$F = \{g_t~|~t \geq 0\} \text{ and } g_t = \diag(e^t, e^{-t/n},\dots, e^{-t/n}).$$
One then invokes a quantitative nondivergence result \cite[Theorem $5.2$]{KM}. In order to address the weighted case, we replace $F$ by $F_{\br}$, where for $\br = (r_1,\dots, r_n)$, we set
$$F_{\br} := \{g(\br)_t~|~t \geq 0\} \text{ and } g(\br)_t = \diag(e^{t}, e^{-r_1 t},\dots, e^{-r_n t}).$$
It remains to apply the arguments in \cite[\S 3]{Kleinbock}, \emph{mutatis mutandis}, to deliver our weighted analogue of [loc. cit., Theorem 4.2].

Assume therefore that $M$ is an $\br$-extremal, smooth, nondegenerate submanifold of an affine subspace $L$. Then for almost every $A \in M$, we have that $\omega_{\br}(A) = 1$. We may therefore apply Corollary \ref{cor:upper} to conclude that for almost every $A \in M$ and every $\btet$ we have $\omega_{\br}(A, \btet) \geq 1$.

To prove the lower bound, we apply the transference principle developed by Beresnevich and Velani \cite{BV2010}; see also \cite{BV2010a} where a simplified exposition of the simultaneous transference principle is given, and additionally \cite{BKM15} where the weighted theory is developed. The proof proceeds via a more or less verbatim repetition of \cite{BV2010}, so we omit it.
\end{proof}

%In fact, the Theorem above should be applicable in a more general scenario (see Theorem $1.3$ and Corollary $1.4$ in \cite{Kleinbock-exponent}). Namely, one takes a $d$-dimensional affine subspace $L$ of $\bR^n$, a friendly measure (see \cite{KLW} for the definition) $\mu$ on $\bR^d$, and $f : \bR^d \to L$ an affine isomorphism. Then in loc. cit. it is shown that  $\omega(f_{*}\mu) = \omega(L)$. In particular, $L$ is extremal if and only if so is $f(\bR^d)$. The case of $\br$-extremality was not considered by Kleinbock but it is very likely that the same methods yield a weighted analogue of the result.

The results of Kleinbock \cite{Kleinbock-exponent}, alluded to above, apply to more general situations than extremality. In particular, it has been shown that arbitrary Diophantine exponents (not just the critical exponent) of affine subspaces are inherited by smooth nondegenerate submanifolds. These results also have counterparts in inhomogeneous approximation. All of this was investigated in the non-weighted and multiplicative cases in \cite{GM2018} by two of the authors of the present article, where it was shown that the Beresnevich--Velani transference principle can be adapted to the non-extremal setting to get an upper bound for inhomogeneous exponents. The corresponding lower bound was then obtained using the results of Bugeaud and Laurent along with more transference results of German \cite{German2}. We can provide a weighted analogue in the non-extremal case, using a similar adaptation of the Beresnevich--Velani transference principle for the upper bound. The weighted analogues of the Bugeaud--Laurent and Dyson theorems, namely theorems \ref{MainThm} and \ref{thm-khin}, together with the trivial relation $\omega_{\br,\bs} \ge \hat{\omega}_{\br,\bs}$, provide the lower bound as follows.

\begin{thm}
 Let $U, \mu, F$ be as in Theorem \ref{thm:BKM}. Suppose that for $\mu$-almost every $A \in F(U)$ we have $\omega_{\bs, \br}(A) = :  \omega.$ Then for $\mu$-almost every $A \in F(U)$ and every $\btet$, we have
 
 \[  \omega \ge \omega_{\bs, \br}(A,\btet) \ge    \frac{(m+n-1)\rho_{\bs}\rho_{\br}(\delta_{\br}+\delta_{\bs}\omega)
   -\rho_{\bs}\delta_{\br}\delta_{\bs}(\omega-1)}{(m+n-1)\rho_{\bs}\rho_{\br}(\delta_{\br}
   +\delta_{\bs}\omega)+\rho_{\br}\delta_{\br}\delta_{\bs}(\omega-1)}.\]
   
\end{thm}

Note that in \cite{GM2018} the trivial relation $\omega \ge \hat{\omega}$ is also used, and the results can be refined by using the recent optimal lower bound $\omega \ge G(\hat{\omega})$ for vectors obtained in \cite{MaMosh}. The corresponding optimal lower bounds for matrices, or in the weighted case, remain open.

\section{Inhomogeneous intermediate exponents} \label{S6}

In this final section, we prove Theorem \ref{ourBV}, first defining $\omega_d(\balp, \btet)$. As a prelude, we briefly discuss Laurent's ``algebraic'' definition \cite[\S2]{Laurent} in the homogeneous setting $\btet = \bzero$, with the same notations as in the introduction. Laurent made the following observations.

\begin{enumerate}
\item Recall from the introduction that to a rational subspace $\cL$ we may associate Grassmannian coordinates $\bX $. Let $\balp \in \bR^n$. Let $\balp' = (1,\balp) \in \bR^{n+1}$ and $\bX' \in \bR^{ {n+1 \choose d+1 } }$ be homogeneous coordinates for $[1:\balp] \in \bP_\bR^n$ and $\bX \in \bP_\bR^{ {n+1 \choose d+1 } -1 }$ respectively. As shown in \cite[Lemma 1]{Laurent}, we have
\[
\d([1:\balp], \cL) = \frac{ | \balp' \wedge \bX'| }{ | \balp'| \cdot |\bX'| }.
\]
This circumvents the need for the extremal definition \eqref{min}. 
\item The Pl\"ucker embedding is known to establish a bijection from $\Gr(d, \bP_\bR^n)$ to the set of non-zero decomposable multivectors in the exterior algebra $\Lam^{d+1}(\bR^{n+1})$, up to homothety.
\item When optimising over $\cL \in \Gr(d,\bP_\bR^n)$, or equivalently over multivectors in $\Lam^{d+1}(\bR^{n+1})$, one may drop the assumption of decomposability---see the remark following \cite[Definition 4]{Laurent}.
\end{enumerate}

\noindent We thus arrive at the following variant of \cite[Definition 4]{Laurent}, which is equivalent to Definition \ref{InterDef} (we have normalised in accordance with Remark \ref{normalisation}).

\begin{defn}\label{def-8-1}
 Let $n \in \bN$, let $d \in \{0,1,\ldots,n-1\}$, and let $\balp'=(1,\balp)$ with $\balp \in \bR^n$. The \emph{$d$th ordinary exponent} $\ome_d(\balp)$ (resp. the \emph{$d$th uniform exponent} $\hat \ome_d(\balp)$) is the supremum of $\ome \in \bR$ such that there exists $\bX \in \Lam^{d+1}(\bZ^{n+1})$ for which
\[
|\bX|^{n \choose d}  \le T, \qquad |\balp' \wedge \bX|^{n \choose d+1}  \le T^{-\ome}
\]
for some arbitrarily large real numbers $T$ (resp. for every sufficiently large real number $T$).
\end{defn}

Before we define inhomogeneous intermediate exponents, we give another equivalent definition of intermediate exponents. 
\begin{defn}\label{def-8-2}
  Let $n \in \bN$, let $d \in \{0,1,\ldots,n-1\}$ and let $\balp\in \bR^n$. The \emph{$d$th ordinary exponent} $\ome_d(\balp)$ (resp. the \emph{$d$th uniform exponent} $\hat \ome_d(\balp)$) is the supremum of $\ome \in \bR$ such that there exists $\bfZ \in \Lam^{d}(\bZ^{n})\setminus\{\bzero\}$ and $\bfY \in \Lam^{d+1}(\bZ^{n})$  for which
\[
|\bfZ|^{n \choose d}  \le T, \qquad |\balp \wedge \bfZ+\bfY|^{n \choose d+1} \le T^{-\ome}
\]
for some arbitrarily large real numbers $T$ (resp. for every sufficiently large real number $T$).
\end{defn}

\begin{proof}[Proof of equivalence of these two definitions]
  We set $\be_0=(1,0,\ldots, 0)\in \bZ^{n+1}$, and identify $\{\bx=(x_0,\ldots,x_n)\in \bZ^{n+1}: x_0=0 \}$ and $\{\bx=(x_0,\ldots,x_n)\in \bR^{n+1}: x_0=0 \}$ with $\bZ^n$ and $\bR^n$, respectively. We have
\[\Lam^{d+1}(\bZ^{n+1})= \be_0\wedge (\Lam^{d}(\bZ^{n}))\oplus \Lam^{d+1}(\bZ^{n}), 
\]
so we can uniquely decompose any $\bX\in \Lam^{d+1}(\bZ^{n+1})$ as $\be_0\wedge \bfZ-\bfY$, with $\bfZ \in \Lam^{d}(\bZ^{n})$ and $\bfY \in \Lam^{d+1}(\bZ^{n})$. This gives
\begin{align*}
  \balp'\wedge \bX &= (\be_0+\balp)\wedge (\be_0 \wedge \bfZ-\bfY)=-\be_0\wedge\bfY+ \balp\wedge \be_0 \wedge \bfZ -\balp\wedge \bfY \\
  &=-(\be_0+\balp)\wedge(\balp \wedge \bfZ+\bfY).
\end{align*}
From the definition \eqref{def-inner} of the inner product on $\Lam^{d+1}(\bR^n)$, we see that $\be_0\wedge (\balp \wedge \bfZ+\bfY)$ is orthogonal to $\balp \wedge(\balp \wedge \bfZ+\bfY)$. Thus it follows that
\[|\balp \wedge \bfZ+\bfY|\le |\balp'\wedge \bX| \le |\balp'| |\balp \wedge \bfZ+\bfY|.\]
Note that $\be_0\wedge \bfZ$ is orthogonal to $\bfY$, hence
\[\max\{|\bfZ|, |\bfY|\}\le |\bX| \le 2\max\{|\bfZ|, |\bfY|\}.\]
The equivalence of Definition \ref{def-8-1} and Definition \ref{def-8-2} follows from the inequalities above.
\end{proof}

We proceed to define inhomogeneous intermediate exponents.
By lexicographically ordering a basis for $\Lam^{d+1}(\bR^{n+1})$, we obtain a canonical isomorphism
\[
\Lam^{d+1}(\bR^{n+1}) \simeq \bR^{n+1 \choose d+1},
\]
which we use to identify the two spaces.

\begin{defn} \label{OurDef} Let $n \in \bN$, let $d \in \{0,1,\ldots,n-1\}$ and let $\balp \in \bR^n$. Let $\btet \in \bR^{n \choose d+1}$.
The \emph{$d$th ordinary exponent} $\ome_d(\balp, \btet)$ (resp. the \emph{$d$th uniform exponent} $\hat \ome_d(\balp, \btet)$) is the supremum of $\ome \in \bR$ such that there exists $\bfZ \in \Lam^{d}(\bZ^{n})\setminus\{\bzero\}$ and $\bfY \in \Lam^{d+1}(\bZ^{n})$ for which
\begin{equation} \label{OurInequalities}
|\bfZ|^{n \choose d}   \le T, \qquad |\balp \wedge \bfZ +\bfY+ \btet|^{n \choose d+1}   \le T^{-\ome}
\end{equation}
for some arbitrarily large real numbers $T$ (resp. for every sufficiently large real number $T$).
\end{defn}

Our definition matches the usual definitions for the $d=0$ (simultaneous) and $d=n-1$ (dual) cases. The lemma below formalises this, and generalises the discussion following \cite[Definition 2]{Laurent}.

\begin{lemma} If we interpret $\btet$ as an element of $\bR^n$ in the former case and $\btet$ as an element of $\bR$ in the latter, then
\[
\omega_0(\balp, \btet) = \omega(\balp, \btet), \qquad \hat \omega_0(\balp, \btet) = \hat \omega(\balp, \btet)
\]
and
\[
\omega_{n-1}(\balp, \btet) = \omega(^t \balp, \btet), \qquad \hat \omega_{n-1}(\balp, \btet) = \hat \omega(^t \balp, \btet)
\]
\end{lemma}

\begin{proof}
By comparing Definition \ref{OurDef} with Definition \ref{def-1}, it is clear that to prove the lemma,  it suffices to compare \eqref{OurInequalities} with  \eqref{equ-def-w}.
When $d=0$, $\bfZ\in \bZ\setminus\{\bzero\}$ and $\bfY\in \bZ^n$. Writing $\bfZ=z$ and $\bfY=\by$, we have $\balp\wedge \bfZ=z\balp$. Now \eqref{OurInequalities} becomes
\[ |z|\le T \qquad |z\balp+\by+\btet|^n \le T^{-\ome},\] which coincides with \eqref{equ-def-w}. 

When $d=n-1$, we have $\bfZ\in \bZ^n\setminus\{\bzero\}$ and $\bfY\in \bZ$. Writing $\bfZ=\bz$, $\bfY=y$ and $\btet=\theta$, we have $\balp\wedge \bfZ=\langle\balp, \bz \rangle$. Now \eqref{OurInequalities} becomes
\[|\bz|^n \le T \qquad |\langle\balp, \bz \rangle+y+\theta| \le T^{-\ome},\] which coincides with \eqref{equ-def-w}.
\end{proof}

We are ready to prove Theorem \ref{ourBV}. The wedge product with $\balp$ defines a linear map from $\Lam^{d}(\bR^n)$ to $\Lam^{d+1}(\bR^n)$. Hence, in view of Definition \ref{OurDef}, Theorem \ref{ourBV} follows from Theorem \ref{BL}; we expound on this below.

Instead of using matrices, we may define our Diophantine exponents for linear maps and their transpose linear maps. Taking the wedge product with $\balp \in \bR^n$ defines linear maps
\[
\bR^N \simeq \Lam^d(\bR^n) \xrightarrow{f = \balp \wedge \cdot} \Lam^{d+1} (\bR^n) \simeq \bR^M
\]
and
\[
\bR^M \simeq \Lam^{n-d-1}(\bR^n) \xrightarrow{g = \balp \wedge \cdot} \Lam^{n-d} (\bR^n) \simeq \bR^N.
\]
These are transposes of one another, up to sign, for if $\bbet \in \Lam^{n-d-1}(\bR^n) \simeq \Lam^{d+1}(\bR^n)^\vee$ and $\bgam \in \Lam^d(\bR^n)$ then
\[
|\bbet \wedge f(\bgam)| = |\bbet \wedge \balp \wedge \bgam| = |g(\bbet) \wedge \bgam|.
\]
Theorem \ref{BL} therefore reveals that
\[ 
\omega_d(\balp,\btet) \geq \cfrac{1}{\hat{\omega}_{n-1-d}(\balp)},
\]
with equality for almost all $\btet$.

%\begin{remark}
%In \cite[\S2]{German}, German defines Diophantine exponents of the \emph{second type}, extending the definition of $\omega_d(\btet)$ to systems of linear forms. Our method also %works in that context.
%\end{remark}

\providecommand{\bysame}{\leavevmode\hbox to3em{\hrulefill}\thinspace}


\begin{thebibliography}{50}

\bibitem{Beresnevich} V. Beresnevich, \emph{Rational points near manifolds and metric Diophantine approximation} Ann. of Math. (2) \textbf{175} (2012), 187--235.

\bibitem{BGGV} V. Beresnevich, A. Ganguly, A. Ghosh and S. Velani, \emph{Inhomogeneous Dual Diophantine Approximation on Affine Subspaces}, Int. Math. Res. Not., doi:10.1093/imrn/rny124, 32 pp.

 \bibitem{BKM15} V. Beresnevich, D. Kleinbock and G. Margulis, \emph{Non-planarity and metric Diophantine approximation for systems of linear forms}, J. Th\'eor. Nombres Bordeaux \textbf{27} (2015), 1--31.

\bibitem{BV2010}
V. Beresnevich and S. Velani, \emph{An inhomogeneous transference principle and Diophantine approximation}, Proc. Lond. Math. Soc. (3) \textbf{101} (2010), 821--851.

\bibitem{BV2010a} V. Beresnevich and S. Velani, \emph{Simultaneous inhomogeneous Diophantine approximations on manifolds}, Fundam. Prikl. Mat. \textbf{16} (2010), 3--17.

\bibitem{BM2017}
P. Bengoechea and N. Moshchevitin, \emph{Badly approximable points in twisted Diophantine approximation and Hausdorff dimension}, Acta Arith.  \textbf{177} (2017), 301--314.

\bibitem{BKLR} Y. Bugeaud, D. Kim, S. Lim and M. Rams, \emph{Hausdorff dimension in inhomogeneous Diophantine approximation}, arXiv:1805.10436.

\bibitem{BL2005}
Y. Bugeaud and M. Laurent, \emph{On exponents of homogeneous and inhomogeneous Diophantine approximation}, Mosc. Math. J. \textbf{5} (2005), 747--766.

\bibitem{Cas1} J. W. S. Cassels, \emph{An introduction to Diophantine Approximation}, Cambridge Tracts in Math. and Math. Phys., vol. 99, Cambridge University Press, 1957.

\bibitem{Cas2} J. W. S. Cassels, \emph{An introduction to the Geometry of Numbers}, Springer Verlag, 1997.

\bibitem{Che2013}
N. Chevallier, \emph{Best simultaneous Diophantine approximations and multidimensional
continued fraction expansions}, Mosc. J. Comb. Number Theory \textbf{3} (2013), 3--56.

\bibitem{Dir} J. P. G. Lejeune Dirichlet, \textit{Verallgemeinerung eines Satzes aus der Lehre ven Kettenbr\"uchen nebst einigen Anwendungen auf die Theorie die Zahlen}, S.-B. Preus Akad. Wiss. (1842), 93--95.

\bibitem{Dyson} F. J. Dyson, \emph{On simultaneous Diophantine approximations}, Proc. Lond. Math. Soc. (2) \textbf{49} (1947), 409--420.

\bibitem{Eve2018}
J.-H. Evertse, \emph{Mahler's work on the geometry of numbers}, arXiv:1806.00356.

\bibitem{Falconer}
K. Falconer, \emph{Fractal geometry. Mathematical foundations and applications}, third edition, Wiley \& Sons, Chichester, 2014.

%\bibitem{German} O. N. German, \emph{Intermediate Diophantine exponents and parametric geometry of numbers}, Acta Arith. \textbf{154} (2012), 79--101.

\bibitem{German2} O. N. German, \emph{On Diophantine exponents and Khintchine's transference principle}, Mosc. J. Comb. Number Theory \textbf{2} (2012), 22--51.

\bibitem{Ghosh} A. Ghosh,  \emph{Diophantine approximation on subspaces of $\mathbb{R}^n$ and dynamics on homogeneous spaces}, Handbook of Group Actions (Vol. IV) ALM 41, Ch. 9, pp. 509--527. Editors Lizhen Ji, Athanase Papadopoulos, Shing-Tung Yau.

\bibitem{GM2018}
A. Ghosh and A. Marnat, \emph{On Diophantine transference principles}, Math. Proc. Camb. Phil. Soc., to appear.

\bibitem{GL} P. M. Gruber and C. G. Lekkerkerker, \emph{Geometry of numbers}, North-Holland Mathematical Library, vol. 37, North-Holland Publishing Co., Amsterdam, 1937.

\bibitem{HM2017}
S. Harrap and N. Moshchevitin, \emph{A note on weighted badly approximable linear forms}, Glasg. Math. J. \textbf{59} (2017), 349--357.

\bibitem{Jar1938}
V. Jarn\'ik, \emph{Zum khintchineschen ``\"Ubertragungssatz''}, Trav. Inst. Math. Tbilissi \textbf{3} (1938), 193--212.

\bibitem{Jar1959}
V. Jarn\'ik, \emph{Eine Bemerkung zum \"{U}bertragimgssatz}, B\"{u}algar. Akad. Nauk Izv. Mat. Inst. \textbf{3} (1959), 169--175.

\bibitem{KL1988} R. Kannan and L. Lov\'asz, \emph{Covering minima and lattice-point-free convex bodies}, Ann. of Math. \textbf{2} (1998) \textbf{128}, 507--602.

\bibitem{Ban1996} W. Banaszczyk, \emph{Inequalities for convex bodies and polar reciprocal lattices in $\bR^n$. II. Application of $K$-convexity},
Discrete Comput. Geom. \textbf{16} (1996), 305--311. 

\bibitem{Khi1926}
A. Ya. Khintchine, \emph{\"Uber eine Klasse linearer diophantischer Approximationen}, Rendiconti Circ. Mat. Palermo \textbf{50} (1926), 170--195.

\bibitem{Kleinbock} D. Kleinbock,  \emph{Extremal subspaces and their submanifolds}, Geom. Funct. Anal. \textbf{13} (2003), 437--466.

\bibitem{Kleinbock-exponent} D. Kleinbock, \textit{An extension of quantitative nondivergence and applications to Diophantine exponents}, Trans. Amer. Math. Soc. \textbf{360} (2008), 6497--6523.

\bibitem{KLW} D. Kleinbock, E. Lindenstrauss and  B. Weiss, \textit{On fractal measures and Diophantine approximation}, Selecta Math. (N.S.) \textbf{10} (2004), 479--523.

\bibitem{KM} D. Kleinbock and G. A. Margulis, \emph{Flows on homogeneous spaces and Diophantine approximation on manifolds}, Ann. of Math. \textbf{148} (1998), 339--360.

\bibitem{KW2008} D. Kleinbock and B. Weiss, \emph{Dirichlet's theorem on Diophantine approximation and homogeneous flows}, J. Mod. Dyn. \textbf{2} (2008), 43--62.

\bibitem{KW2010} D. Kleinbock and B. Weiss, \emph{Modified Schmidt games and Diophantine
approximation with weights}, Adv. Math. \textbf{223} (2010), 1276--1298.

\bibitem{Lag1} J. C. Lagarias, \emph{Best simultaneous Diophantine approximations. I. Growth rates of best approximation denominators}, Trans. Amer. Math. Soc. \textbf{272} (1982), 545--554.

\iffalse
\bibitem{Lag2} J. C. Lagarias, \emph{Best simultaneous Diophantine approximations. II. Behavior of consecutive best approximations}, Pacific J. Math. \textbf{102} (1982), 61--88.
\fi

\bibitem{LagAus} J. C. Lagarias, \emph{Best Diophantine approximations to a set of linear forms}, J. Austral. Math. Soc. Ser. A \textbf{34} (1983), 114--122.

\bibitem{Laurent} M. Laurent, \emph{On transfer inequalities in Diophantine approximation.} Analytic number theory, 306--314, Cambridge Univ. Press, Cambridge, 2009.

\bibitem{LSS} S. Lim, N. de Sax\'{c}e and U. Shapira, \emph{Dimension bound for badly approximable grids}, Int. Math. Res. Not., to appear.

\bibitem{Mah1932} K. Mahler, \emph{\"{U}ber das Mass der Menge aller S-Zahlen}, Math. Ann. \textbf{106} (1932), 131--139.

\bibitem{Mah1939}
K. Mahler, \emph{Ein \"Ubertragungsprinzip f\"ur konvexe K\"orper}, Math. \v{C}asopis \textbf{68} (1939), 93--102.

\bibitem{Mah55}
K. Mahler, \emph{On compound convex bodies. I.}, Proc. London Math. Soc. (3) \textbf{5} (1955), 358--379.

\bibitem{MaMosh}
A. Marnat and N. G. Moshchevitin, \emph{An optimal bound for the ratio between ordinary and uniform exponents of Diophantine approximation}, arXiv:1802.03081.

\bibitem{Philippon}
P. Philippon, \emph{Crit\`eres pour l'ind\'ependance alg\'ebrique}, Inst. Hautes \'Etudes Sci. Publ. Math. \textbf{64} (1986), 5--52.

\bibitem{Rog} C. A. Rogers, \textit{The signature of the errors of some Diophantine approximations}, Proc. Lond. Math. Soc. \textbf{52} (1951) , 186--190.

\bibitem{Roy} D. Roy, \emph{Diophantine approximation in small degree}, Number Theory, CRM Proc. Lecture Notes, \textbf{36} (2004), 269--285.

\bibitem{SchH} W. M. Schmidt, \textit{On heights of algebraic subspaces and Diophantine approximations}, Ann. of Math. (2) \textbf{85} (1967), 430--472.

\bibitem{Schmidt} W. M. Schmidt, \emph{Diophantine approximation}, Lecture Notes in Mathematics, vol. 785,  Springer, Berlin, 1980.

\bibitem{Sie1989}
C. L. Siegel, \emph{Lectures on the Geometry of Numbers}, Springer--Verlag, Berlin--Heidelberg--New York 1989.

\bibitem{Sp1} V. Sprind\v{z}uk, \emph{More on Mahler's conjecture}, Dokl. Akad. Nauk SSSR \textbf{155} (1964), 54--56 (Russian); English transl. in Soviet Math. Dokl. \textbf{5} (1964), 361--363.

\bibitem{Sp2} V. Sprind\v{z}uk, \emph{Mahler's problem in metric number theory}, Translations of Mathematical Monographs, vol. 25, Amer. Math. Soc., Providence, RI, 1969.

\bibitem{Sp3} V. Sprind\v{z}uk, \emph{Achievements and problems in Diophantine approximation theory}, Russian Math. Surveys \textbf{35} (1980), 1--80.

\bibitem{Vor} G. F. Voronoi, \textit{On one generalization of continued fractions algorithm}, USSR Ac. Sci., v.1 (1952), 197--391.



\iffalse
\bibitem{TV2008}
T. Tao and V. Vu, \emph{John-type theorems for generalized arithmetic progressions and iterated sumsets}, Adv. Math. \textbf{219} (2008), 428--449.
\fi

\end{thebibliography}
\end{document}